\documentclass[11pt,a4paper,intlimits,reqno]{amsart}

\usepackage{amsmath}
\usepackage{amsopn,amsfonts,amssymb,amsthm}

\usepackage{mathrsfs}
\usepackage[T1]{fontenc}
\usepackage[english]{babel}
\usepackage[utf8]{inputenc}
\usepackage{epsfig}
\input{epsfx}
\usepackage{graphicx}

\usepackage{fancyhdr}
\fancypagestyle{plain}{%
\fancyhead{} % usuń p. górne na stronach pozbawionych
}

\newtheorem{thm}{Theorem}[section]
\newtheorem{prop}[thm]{Proposition}
\newtheorem{lem}[thm]{Lemma}

\newtheorem{cor}[thm]{Corollary}

\numberwithin{equation}{section}

\newcommand{\N}{\mathbb{N}}

\newcommand{\C}{\mathbb{C}}

\newcommand{\R}{\mathbb{R}}

\DeclareMathOperator{\im}{Im}
\DeclareMathOperator{\re}{Re}
\DeclareMathOperator{\dist}{dist}
\DeclareMathOperator{\diam}{diam}
\DeclareMathOperator{\HD}{HD}

\DeclareMathOperator{\erior}{Int}

\newcommand{\sms}{\setminus}
\newcommand{\leeq}{\leqslant}
\newcommand{\greq}{\geqslant}

\newcommand{\ve}{\varepsilon}
\newcommand{\la}{\lambda}
\newcommand{\mc}{\mathcal}
\newcommand{\vp}{\varphi}
\newcommand{\de}{\delta}
\newcommand{\dt}{\divideontimes}

\begin{document}

\title[On the directional derivative of the Hausdorff dimension]{On the directional derivative of the Hausdorff dimension of quadratic polynomial Julia sets at -2}

\author{Ludwik Jaksztas}

\address{Faculty of Mathematics and Information Science, Warsaw University of Technology, ul. Koszykowa 75, 00-662 Warsaw, Poland, ORCID: 0000-0002-9283-8841}
\email{jaksztas@impan.pl}
\thanks{The author was partially supported by National Science Centre Grant 2019/33/B/ST1/00275 (Poland)}
\subjclass[2020]{Primary 37F44, Secondary 37F35}
\keywords{Hausdorff dimension, Julia sets, quadratic family}

\begin{abstract}
Let $d(\delta)$ denote the Hausdorff dimension of the Julia set of the polynomial $f_\delta(z)=z^2-2+\delta$.

In this paper we will study the directional derivative of the function $d$ along directions landing at the parameter $0$, which corresponds to $-2$ in the case of family $p_c(z)=z^2+c$. We will consider all directions, except the one $\delta\in\R^+$, which is inside the Mandelbrot set.

We will prove asymptotic formula for the directional derivative of $d$. Moreover, we will see that the derivative is negative for all directions in the closed left half-plane. Computer calculations show that it is negative except a cone (with opening angle approximately $74^\circ$) around $\R^+$.
\end{abstract}

\maketitle

%\tableofcontents

\section{Introduction}\label{sec:introd}

Let $f$ be a polynomial in one complex variable of degree at least $2$. The \emph{filled-in Julia set} $\mc K_f$ we define as the set of all points that do not escape to infinity under iteration of $f$, i.e.
   $$\mc K_f=\{z\in\C:f^n(z)\nrightarrow\infty\}.$$
It is a compact set whose boundary is called the \emph{Julia set}. So, let us write
$$\mc J_f:=\partial\mc K_f.$$
Instead of the classical quadratic family $p_c(z)=z^2+c$, we will consider for technical reasons
$$f_\delta(z)=z^2-2+\delta,$$
i.e. $\de=c+2$. We will use the following abbreviations:
$$\mc J_\de:=\mc J_{f_\de},  \quad\quad \mc K_\de:=\mc K_{f_\de}.$$
Let $d(\de):=\HD(\mc J_\de)$ denote the Hausdorff dimension of the Julia set.
In this paper, we will deal with the function
$$\de\mapsto d(\de).$$

We define the Mandelbrot set (for the family $f_\de$) as follows
$$\mc M:=\{\delta\in\C:f_\delta^n(0)\nrightarrow\infty\}.$$

Recall that a polynomial $f:\overline\C\rightarrow\overline\C$ is called
\emph{hyperbolic (expanding)} if there exists $n\greq1$ such that $|(f^n)'(z)|>1$ for every $z\in\mc J_f$.

The function $d(\de)$ is real-analytic on each hyperbolic component of $\erior\mc M$ (consisting of parameters
related to hyperbolic maps) as well as on the exterior of $\mc M$ (see \cite{R}).
For $\alpha\in[0,2\pi)$, let us write
$$\mc R(\alpha):=\{z\in\C^*:\alpha=\arg z\}.$$
We will study properties of the function $d(\de)$ when the parameter $\de\notin \mc M$ tends to $0\in\partial\mc M$ along the ray $\mc R(\alpha)$. So, we will consider all directions except $\mc R(0)$ which is inside $\mc M$.

Parameter $\de=0$ corresponds to $c=-2$, so this is Misiurewicz's parameter. Note that a parameter is called Misiurewicz's if the critical point is strictly preperiodic (that is preperiodic but not periodic). In our casse $0\mapsto-2\mapsto2$, where $2$ is a fixed point. Moreover the Julia set $\mc J_{0}$ is equal to the interval $[-2,2]$ (see for example \cite{CG}).

J. Rivera-Letelier  proved in \cite{R} the following:\vspace{2.5mm}
\newline
\textbf{Theorem R-L.} \emph{There exists $C_0>0$ such that if $\de_n\rightarrow0$ and}
$$\re(\de_n)\leeq0\quad\textrm{or}\quad|\im(\de_n)|>C_0|\re(\de_n)|^{3/2},$$
\emph{then $\de_n\notin\mc M$ and $d(\de_n)\rightarrow d(0)=1$.}
\vspace{2mm}

An easy consequence of the above Theorem is that $d(\de)$ converges to $1$ along any ray $\mc R(\alpha)$, $\alpha\in(0,2\pi)$.

In \cite{FJW}, A. Fan, Y. Jiang and J. Wu, gave estimate of $d$ restricted to the real line. They
proved that: \vspace{2.5mm}
\newline
\textbf{Theorem FJW.} \emph{There exists constants $K>0$ and $\de_1<0$ such that}
   $$1-K^{-1}\sqrt{|\de|}\leeq d(\de)\leeq1-K\sqrt{|\de|},$$
for all $\de_1\leeq\de<0$.
\vspace{2mm}

So, one can expect that derivative $d'(\de)$ tends to $\infty$ like $1/\sqrt{|\de|}$, when $\de\rightarrow0^-$. In this paper we prove much more general result.
In order to state our main theorem we need two definitions.

Let $F$ be a real function defined on a domain $U\in\C$. If $z\in U$ and $v\in\C^*$, then
   $$F'_v(z):=\lim_{h\rightarrow0}\frac{F(z+hv)-F(z)}{h},$$
if the limit exists.

For $\alpha\in(0,\pi]$ we write
\begin{equation}\label{eq:omega}
\Omega_{-2}(\alpha):=\frac{1}{\sqrt6\,\pi\log2}\bigg(\cos\alpha-\frac12\sqrt{\sin\alpha}\int_{\alpha}^\pi\sqrt{\sin x}\,dx\bigg).
\end{equation}
If $\alpha\in(\pi,2\pi)$ then we define $\Omega_{-2}(\alpha):=\Omega_{-2}(2\pi-\alpha)$.

\begin{figure}[!ht]
   \centering
        \includegraphics[scale=0.75]{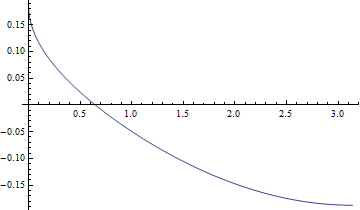}
        \caption{Graph of the function $\Omega_{-2}(\alpha)$ for $\alpha\in(0,\pi]$.}
\end{figure}

The main Theorem in this paper is:

\begin{thm}\label{thm:minus2}
For every $\alpha\in(0,2\pi)$ we have
\begin{equation*}
\lim_{\de\rightarrow0}\sqrt{|\de|}\cdot d_v'(\de)= \Omega_{-2}(\alpha),
\end{equation*}
where $\alpha=\arg\de$ and $v=e^{i\alpha}$.
\end{thm}

Because of symmetry, it is enough to prove the Theorem for $\alpha\in(0,\pi]$.
The function $\Omega_{-2}(\alpha)$ is obviously negative for $\alpha\in[\pi/2,\pi]$, whereas is positive for small $\alpha$.
Moreover it is easy to check that $\frac{\partial}{\partial\alpha}\Omega_{-2}(\alpha)<0$ for $\alpha\in(0,\pi/2)$. Therefore there exists a unique $\alpha_0\in(0,\pi/2)$ such that $\Omega_{-2}(\alpha_0)=0$. Thus we have $d'_v(\de)\rightarrow+\infty$ for every $\alpha\in(0,\alpha_0)$, and $d'_v(\de)\rightarrow-\infty$ for every $\alpha\in(\alpha_0,\pi]$.

A numerically made picture of the graph of $\Omega(\alpha)$ (see Figure 1.) suggests that it is a decreasing function on whole interval $(0,\pi)$, whereas $\alpha_0$ is slightly greater than $\pi/5$ (close to $37^\circ$).

In the real case (i.e. $\alpha=\pi$), we see that
$$\lim_{\de\rightarrow0}\sqrt{|\de|}\cdot d_{-1}'(\de)=\frac{-1}{\sqrt6\,\pi\log2}.$$
Note that this is the directional derivative, whereas the derivative of $d(\de)$ has opposite sign. So, integrating we obtain:
\begin{cor}
If $\de\in\R^-$, then for $|\de|$ small we have
$$d(\de)=1-\frac{\sqrt6}{3\,\pi\log2}\sqrt{|\de|}+o(\sqrt{|\de|}).$$
\end{cor}
\vspace{2mm}
\noindent
\textbf{Remark. Estimates along the ray $\mc R(0)$ (inside the Mandelbrot set).}

Recently N. Dobbs, J. Graczyk and N. Mihalache proved in \cite{DGM} the following significant result: %concerning positive real case (i.e. parameters from the Mandelbrot set)
\vspace{2.5mm}
\newline
\textbf{Theorem DGM.} \emph{There exists constants $\kappa_*>0$ and $\de_0>0$ such that for every $\de\in(0,\de_0]$}
   $$d(\de)\greq1+\kappa_*\sqrt{\de}.$$
%\vspace{0.5mm}

In that paper also an upper bound is proved, which holds on a "large" set of parameters.

The fact that we can pass with $\Omega_{-2}(\alpha)$ to the limit as $\alpha\rightarrow0^+$, that is
$$\lim_{\alpha\rightarrow0^+}\Omega_{-2}(\alpha)=\frac{1}{\sqrt6\,\pi\log2},$$
leads to the following:\vspace{2.5mm}
\newline
\textbf{Conjecture.}
\emph{The supremum over all constants $\kappa_*>0$ for which there exists $\de_0>0$ such that the statement of Theorem DGM holds, is equal to $\sqrt6/(3\,\pi\log2)$. In particular}
   $$d(\de)\greq1+\frac{\sqrt6}{3\,\pi\log2}\sqrt{\de}-o(\sqrt{\de}),$$
\emph{for small $\de\in\R^+$.}
\vspace{2mm}

Very recently A. Dudko, I. Gorbovickis and W. Tucker computed in \cite{DGT} rigorous lower bounds of $d(\de)$ for $\de\in[0,4]$ (in our parametrization). They estimated $d(\de)$ on some intervals and also on some additional parameters.
%As far as I know this is the only rigorous estimate of $d(\de)$ for $\de$ close to $0$.

Now, using lower bounds of $d(\de)$ for parameters (which are greater than estimates for intervals containing them) we make the following observation: For $\de_n\approx n\cdot0.004$, where $1\leeq n\leeq25$, precise values of estimates (which we received from the Authors) lead to
   $$d(\de_n)>1+0.362\sqrt{\de_n},$$
whereas $\sqrt6/(3\,\pi\log2)\approx0.375$.

\

This is the first paper concerning the derivative of the dimension function close to a non-parabolic parameter. On the other hand this is the second paper about the directional derivative after \cite{Jiii}, where the derivative close to the parameter $c=1/4$ was studied (parabolic parameter with one petal).

As in \cite{Jiii} we will us the formula for the derivative (quotient of two integrals, see Proposition \ref{prop:d}). The integral from the denominator is equal to the Lyapunov exponent (if the measure is normalized) and tends to a positive constant. The main problem is to estimate integral from the numerator.

Note that our situation is much different than in \cite{Jiii}. In our case a crucial role is played by a part of the Julia set that is close to $0$, and we will have to rescale the set $\mc J_\de$ in order to study a location of $\mc J_\de$ and behaviour of conformal and invariant measures in detail.

Notation: if $z\in\C\sms\R$, then $\sqrt{z}$ denotes square root with positive real part. If $z\in\R^-$ then we assume that $\sqrt{z}$ has positive imaginary part.

\section{Thermodynamic formalism}\label{sec:formalism}

The main goal of this section is to establish a formula for directional derivative of the Hausdorff dimension (see Proposition \ref{prop:d})
(cf. \cite[Section 2]{Jiii}).

Because of hyperbolicity, the Julia set moves holomorphically over every simply connected subset of $\C\sms \mc M$. In particular, we can take the set $\mc U$, which is connected component of $B(0,1/5)\sms \mc M$ containing $\de_0=-1/10$.

Thus, there exists family of injections $\vp_\de:\mc J_{\de_0}\rightarrow\hat\C$, $\de\in\mc U$ such that $\vp_{\de_0}=\textrm{id}_{\mc J_{\de_0}}$, $\vp_{\de}(\mc J_{\de_0})=\mc J_\de$ and
   $$\vp_\de\circ f_{\de_0}(s)=f_\de\circ \vp_\de(s),$$
for $s\in\mc J_{\de_0}$ (i.e. $\vp_\de$ conjugates $f_{\de_0}$ to $f_\de$). Moreover, the map $\de\mapsto\vp_\de(s)$ is holomorphic, whereas for every $\de\in\mc U$ the map $\vp_\de$ extends to a quasiconformal (so also H\"{o}lder) map of the sphere to itself (see \cite{Hub}).

Note that for $\de=0$, there exists function $\vp_0$ that semiconjugates $f_{\de_0}$ to $f_0$. A point $z$ has two preimages under $\vp_0$ iff $f_0^n(z)=0$ for some $n\greq0$ (see \cite[Theorem 1.6]{CK}).

Now we use the thermodynamic formalism, which holds for hyperbolic rational maps (see \cite{PU}).
Let $X=\mc J_{\de_0}$, $T=f_{\de_0}$, and let $\phi:X\rightarrow\R$ be a potential function of the form
$\phi=-\tau\log|f_\delta'(\vp_\delta)|$, for $\delta\in \mc U$ and $\tau\in\R$.

{\em The topological pressure} can be defined as follows:
   $$P(T,\phi):=\lim_{n\rightarrow\infty}\frac{1}{n}\log\sum_{\overline s\in{T^{-n}(s)}}e^{S_n(\phi(\overline s))},$$
where $S_n(\phi)=\sum_{k=0}^{n-1}\phi\circ T^k$. The limit exists and does not depend on $s\in\mc J_{\de_0}$ (see \cite[Corollary 12.5.13]{PU}).
If $\phi=-\tau\log|f_\delta'(\vp_\delta)|$ and $\vp_\delta(\overline s)=\overline z$, then $e^{S_n(\phi(\overline
s))}=|(f_\delta^n)'(\overline z)|^{-\tau}$, hence
   $$P(T,-\tau\log|f_\delta'(\vp_\delta)|)=\lim_{n\rightarrow\infty}\frac{1}{n}\log\sum_{\overline z\in{f_\delta^{-n}(z)}}|(f_\delta^n)'(\overline z)|^{-\tau}.$$
The function $\tau\mapsto P(T,-\tau\log|f_\delta'(\vp_\delta)|)$ is strictly decreasing from $+\infty$ to $-\infty$. So, there exists a
unique $\tau_0$ such that $P(T,-\tau_0\log|f_\delta'(\vp_\delta)|)=0$. By Bowen's Theorem (see \cite[Corollary 9.1.7]{PU} or
\cite[Theorem 5.12]{Z}) we obtain
  \begin{equation*}\label{eq:acm}
    \tau_0=d(\delta).%=\HD(\mc J_\delta).
  \end{equation*}
Thus, we have $P(T,-d(\delta)\log|f_\delta'(\vp_\delta)|)=0$. Write $\phi_\delta:=-d(\delta)\log|f_\delta'(\vp_\delta)|$.

{\em The Ruelle operator} or {\em the transfer operator} $\mathcal{L}_{\phi}:C^0(X)\rightarrow C^0(X)$, is defined as
  \begin{equation}\label{eq:L}
    \mathcal{L}_{\phi}(u)(s):=\sum_{\overline s\in{T^{-1}(s)}}u(\overline s)e^{\phi(\overline s)}.
  \end{equation}
The Perron-Frobenius-Ruelle theorem \cite[Theorem 4.1]{Z} asserts that $\beta=e^{P(T,\phi)}$ is a single
eigenvalue of $\mathcal{L}_{\phi}$ associated to an eigenfunction $\tilde h_{\phi}>0$. Moreover there exists a
unique probability measure $\tilde \omega_{\phi}$ such that $\mathcal{L}_{\phi}^*(\tilde
\omega_{\phi})=\beta\tilde\omega_{\phi}$, where $\mathcal{L}_{\phi}^*$ is dual to $\mathcal{L}_{\phi}$.

For $\phi=\phi_\delta$ we have $\beta=1$, and then $\tilde \mu_{\phi_\delta}:=\tilde
h_{\phi_\delta}\tilde\omega_{\phi_\delta}$ is a $T$-invariant measure called an equilibrium state (if it is normalized).
We denote by $\tilde{\omega_\delta}$ and $\tilde{\mu_\delta}$ the measures $\tilde\omega_{\phi_\delta}$ and
$\tilde\mu_{\phi_\delta}$ respectively (measures supported on $\mc J_{\de_0}$). Next, we take
$\mu_{\delta}:=(\vp_{\delta})_*\tilde{\mu_{\delta}}$, and $\omega_{\delta}:=(\vp_{\delta})_*\tilde{\omega_{\delta}}$ (measures supported on the Julia set $\mc J_\delta$).

So, the measure $\mu_\delta$ is $f_\delta$-invariant, whereas $\omega_\delta$ (after normalization) is called {\em $f_\delta$-conformal measure with exponent
$d(\delta)$}, i.e. $\omega_\delta$ is a Borel probability measure such that for every Borel subset $A\subset\mc J_\delta$,
  \begin{equation}\label{eq:SP}
    \omega_\delta(f_\delta(A))=\int_A|f_\delta'|^{d(\delta)}d\omega_\delta,
  \end{equation}
provided $f_\delta$ is injective on $A$.

However we will not assume that $\mu_\delta$ and $\omega_\delta$ are probability measures. Because $\mc J_0=[-2,2]$ we will take $\mu_\delta(\mc J_\de)=\omega_\delta(\mc J_\de)=4$.

It follows from \cite[Proposition 6.11]{Z} or \cite[Theorem 5.6.5]{PU} that for every H\"{o}lder continuous functions $\psi,\hat\psi:X\rightarrow\R$, at every $t\in\R$, we have
  \begin{equation*}\label{eq:SP1}
    \frac{\partial}{\partial t}P(T,\psi+t\hat\psi)=\int_X \hat\psi \:d\tilde{\mu}_{\psi+t\hat\psi}.
  \end{equation*}

Let us consider parameters of the form $\delta=tv$, where $v=e^{i\alpha}$, $\alpha=\arg\de$, and $t>0$.
%Let us consider parameters of the form $\delta=te^{i\alpha}=tv$, where $t>0$.
Since $\tau=d(tv)$ is the unique zero of the pressure function, for the potential $\phi=-\tau\log|f'_{tv}(\varphi_{tv})|$, the implicit function theorem combined with the above formula leads to (see \cite[Proposition 2.1]{HZi} or \cite[Proposition 2.1]{J}):

\begin{prop}\label{prop:d}
If $\alpha\in(0,2\pi)$ and $v=e^{i\alpha}$, then for every $t>0$ such that $tv\notin\mc M$ we have
\begin{equation}\label{eq:d}
d_v'(tv)=-d(tv)
\frac{\int_{\mc J_{\de_0}}\frac{\partial}{\partial
t}\log|f_{tv}'(\vp_{tv})|d\tilde\mu_{tv}}{\int_{\mc J_{\de_0}}\log|f_{tv}'(\vp_{tv})|d\tilde\mu_{tv}}.
\end{equation}
\end{prop}

\section{Conjugation close to the fixed point}\label{sec:conj}

Suppose $p$ is a repelling (or attracting) fixed point of a holomorphic function $f$.
It is well known fact that $f$ is conjugated to $z\mapsto f'(p)z$ in a neighborhood of $p$.
Now we recall how to construct the conjugation.

The polynomials $f_\de$, $\de\in\mc U$ have two repelling fixed points. We will consider
\begin{equation}\label{eq:1/27}
p_\delta:=\frac12+\frac32\sqrt{1-\frac49\delta}=2-\frac13\,\delta-\frac{1}{27}\,\delta^2+O(\delta^3),
\end{equation}
which becomes $p_0=2$ for $\de=0$.
So, the multiplier $\la_\de$ at $p_\de$ is equal to
$$\la_\de=f_\de'(p_\de)=2p_\de.$$

In order to construct the conjugation $\Phi_{\delta}$, we first "move" $p_\de$ to $0$. Put
$$\hat f_\de(z):=f_\de(z+p_\de)-p_\de.$$
Indeed, $0$ is the fixed point of $\hat f_\de$. We will write $\hat{\mc J_\de}:=\mc J_{\hat f_\de}$.

The sequence $\hat\Phi_{n,\delta}$ that converges to the function $\hat\Phi_{\delta}$, which conjugates $\hat f_\de$ to $z\mapsto \la_\de z$ (i.e. $\hat\Phi_{\delta}\circ \hat f_\de=\lambda_\de\hat\Phi_{\delta}$), we define as follows:
$$\hat\Phi_{n,\delta}(z):=\lambda_\delta^n\hat f^{-n}_\delta( z)=\lambda_\delta^n(f^{-n}_\delta(p_\delta+z)-p_\delta).$$
Note that
\begin{equation}\label{eq:Phi-1}
   \hat\Phi^{-1}_{n,\delta}(z)=\hat f^{n}_\delta\Big(\frac{ z}{\lambda_\delta^n}\Big)=f^{n}_\delta\Big(p_\delta+\frac{ z}{\lambda_\delta^n}\Big)-p_\delta.
\end{equation}
We have
$$\hat\Phi_{n,\delta}(z)\rightrightarrows\hat\Phi_{\delta}(z),$$
and the convergence is uniform independently of the parameter $\de\in B(0,r_{\vartriangle})$, where $ z\in B(0,r_{\textrm{z}})$, for some $r_\vartriangle>0$, $r_{\textrm{z}}>0$. Thus, $\hat\Phi_\de$ depends analytically on $\de\in B(0,r_\vartriangle)$ (see for example \cite[Chapter 2]{CG}).

Possibly changing $r_{\textrm{z}}>0$, we can assume that both functions $\hat\Phi_\de$ and $\hat\Phi^{-1}_\de$ are defined on $B(0,r_{\textrm{z}})$. Next, because $\hat\Phi_\de'(0)=1$, we can also assume that these functions are not too far from the identity map, namely
\begin{equation}\label{eq:PP}
\Big|\frac{\hat\Phi_\de(z)}{z}-1\Big|<\frac12,\quad\textrm{and}\quad\Big|\frac{\hat\Phi_\de^{-1}(z)}{z}-1\Big|<\frac12,
\end{equation}
for $z\in B(0,r_{\textrm{z}})$ and $\de\in B(0,r_{\vartriangle})$.

%Since $\hat\Phi_\de'(0)=1$ we can assume that
%\begin{equation*}
% \begin{array}{cc}
%   \hat\Phi_\de(B(0,r_{\textrm{z}}/2))\subset B(0,r_{\textrm{z}}),\\
%   \hat\Phi^{-1}_\de(B(0,r_{\textrm{z}}/2))\subset B(0,r_{\textrm{z}}),
% \end{array}
%\end{equation*}
%for $\de\in B(0,r_\vartriangle)$.

Finally, we take $\Phi_{\delta}(z):=\hat\Phi_{\delta}(z-p_\de)$, and then we have
\begin{equation}\label{eq:pfp}
\Phi_{\delta}\circ f_\de=\lambda_\de\Phi_{\delta},
\end{equation}
where the functions $\Phi_\de$ are defined on $B(p_\de, r_\textrm{z})$.
Possibly changing $r_\textrm{z}$ and $r_\vartriangle$, we can also assume that $\Phi_\de$ are defined on $B(2,r_\textrm{z})$ for $\de\in B(0,r_\vartriangle)$.

\section{The Julia sets $\mc J_\de$ for $\de$ close to 0}\label{sec:position}

\subsection{The Hausdorff metric}\label{sec:Hm}
If $X$ and $Y$ are two non-empty compact subsets of $\C$, then their Hausdorff distance is defined as follows:
   $$d_H(X,Y)=\max\Big(\sup_{x\in X}\dist(x,Y),\,\sup_{y\in Y}\dist(y,X)\Big).$$

Now we state an important Theorem from \cite[Section 5]{D}.
\begin{thm}\label{thm:dH}
If $\mc K_{\hat\de}$ has empty interior, that is $\mc K_{\hat\de}=\mc J_{\hat\de}$, then the function
$$\de\mapsto\mc J_{\de},$$
is continuous at $\hat\de$, as the function from $\C$ into the space of non-empty compact subsets of $\C$ equipped with the Hausdorff metric.
\end{thm}

\begin{cor}\label{cor:dH}
If $\de\rightarrow0$, then
$$\mc J_\de\rightarrow\mc J_0=[-2,2],$$
in the space of non-empty compact subsets of $\C$ equipped with the Hausdorff metric.
\end{cor}

\subsection{Conjugations and location of $\mc J_\de$}\label{sse:ce}
It is known fact that the function $x\mapsto \frac{2}{\pi}\arcsin x$
conjugates polynomial $g(x)=-2x^2+1$ defined on $[-1,1]$, to the tent map $T:[-1,1]\rightarrow[-1,1]$ defined by $T(x)=-2|x|+1$ (see for example \cite[Chapter II]{MS}).

Since $x\mapsto-x/2$ conjugates $f_0$ to $g$, let us write
$$\vp_v(x)=\frac4\pi\arcsin\Big(\frac x2\Big),\qquad\vp_v^{-1}(x)=2\sin\Big(\frac\pi4x\Big),$$
and next
$$
V(x)=\left\{ \begin{array}{ll}
2x-2 &\textrm{ if }\quad x\in[0,2],\\
-2x-2&\textrm{ if }\quad x\in[-2,0).
\end{array} \right.
$$
Thus, $\vp_v$ conjugates $f_0$ to $V$. Hence
$$f_0^n(x)=\vp_v^{-1}\circ V^n\circ \vp_v(x),\qquad x\in[-2,2],\; n\in\N.$$
Differentiating the above equality, and using the fact that $(\vp_v^{-1})'( V\circ \vp_v)=(\vp_v^{-1})'(\vp_v\circ f_0)$, we can get
\begin{equation}\label{eq:(fn)'}
|(f_0^n)'(x)|=2^n\sqrt{\frac{4-(f_0^n(x))^2}{4-x^2}}.
\end{equation}

%\begin{equation*}\label{eq:sumas1}
%\frac{d}{d\delta}z(\delta)=-\sum_{k=1}^{m}
%\frac{1}{(f_\delta^k)'(z(\delta))}+\frac{\frac{d}{d\delta}f_\delta^m(z(\delta))}{(f_\delta^m)'(z(\delta))}.
%\end{equation*}

\

The function $F(z)=z+\frac1z$ is a semiconjugation between $z\mapsto z^2$ and $f_0$. Image of the unit circle under $F$ is equal to the Julia set $\mc J_0=[-2,2]$. Let us consider the family of circles $S(0,r)$ where $r>1$. Then, the images $F(S(0,r))=:\xi_r$ are ellipses, which can be parameterized as follows:
$$\xi_r(t):re^{it}+\frac1re^{-it}.$$
For every $r>1$, the focuses are $\pm2$. Note that
$$f_0(\xi_r(t))=\xi_{r^2}(2t).$$
If $z=x+iy$, then we get
$$
\xi_r(t):\left\{ \begin{array}{ll}
x(t)=(r+1/r)\cos t,\\
y(t)=(r-1/r)\sin t.
\end{array} \right.
$$
Thus the semi-minor and semi-major axes are equal to $(r-1/r)$ and $(r+1/r)$ respectively.

Filled ellipses will be denoted by $E_r$, i.e.
$$E_r:\frac{x^2}{(r+1/r)^2}+\frac{y^2}{(r-1/r)^2}\leeq1.$$
Note that for $r\greq1$, we have
\begin{equation}\label{eq:rr}
    r+\frac1r\leeq2+(r-1)^2, \quad\textrm{and}\quad r-\frac1r\leeq 2(r-1).
\end{equation}
%for $r\greq1$.

Now we give some estimates concerning precise location of the sets $\mc J_\de$.
\begin{lem}\label{lem:sqrt|de|}
For every $\de\neq0$ we have
$$\mc J_\de\subset E_{1+\sqrt{|\de|}}\subset\big\{z\in\C:|\im z|\leeq 2\sqrt{|\de|}\big\}.$$
\end{lem}

\begin{proof}
Let $r=1+\kappa\sqrt{|\de|}$, where $\kappa\greq1$.
It is easy to prove that distance between ellipses $\xi_r$ and $\xi_{r^2}$ is equal to semi-major axis difference, namely
$$\Big(r^2+\frac{1}{r^2}\Big)-\Big(r+\frac{1}{r}\Big)=\big(r-1\big)^2\Big(1+\frac{1}{r}+\frac{1}{r^2}\Big)>\big(r-1\big)^2=\kappa^2|\de|. $$
Because
$$f_\de(\xi_r(t))=\xi_{r^2}(2t)+\de,$$
we see that $E_r$ is included in region bounded by $f_\de(\xi_r(t))$, whereas exterior of $E_r$ is mapped onto exterior of  $f_\de(\xi_r(t))$. So, we conclude that if $z\notin E_{1+\sqrt{|\de|}}$, then $f_\de^n(z)\rightarrow\infty$, hence
$$\mc J_\de\subset E_{1+\sqrt{|\de|}}.$$

Since $E_r$ is contained in $\{z\in\C:|\im z|\leeq r-1/r\}$, and for $r=1+\sqrt{|\de|}$ we have $r-1/r\leeq2\sqrt{|\de|}$ (cf. (\ref{eq:rr})), and the statement follows.
\end{proof}

\begin{lem}\label{lem:|de|}
There exists $\eta>0$ such that if $z\in \mc J_\de$, then
\begin{enumerate}
\item\label{lit:|de|1}
\begin{equation*}
\begin{array}{ll}
      |\im z-\im p_\de|<|\de|^{17/16}\quad\textrm{provided}\quad\re z\greq\re p_\de-|\de|^{15/16}, \\
      |\im z+\im p_\de|<|\de|^{17/16}\quad\textrm{provided}\quad\re z\leeq-\re p_\de+|\de|^{15/16},
\end{array}
\end{equation*}
\item\label{lit:|de|2}
$$|\re z|<\re p_\de +|\de|^2,$$
\end{enumerate}
where $0<|\de|<\eta$.
\end{lem}

\begin{proof} %For $\alpha=\pi$ the statement is obvious, thus we can assume that $\alpha\neq\pi$.
Proof will be carried out for $\hat f_\de$ and $z\in\hat {\mc J_\de}$ (i.e. $p_\de$ is moved to $0$).

The curves
   $$\gamma_{\de,\beta}(t)=e^{i\beta}e^{t\log \la_\de},$$
where $\beta\in\R$,
are invariant under $z\mapsto \la_\de z$, so the trajectories of $\hat f_\de$ lies on  $\hat \Phi^{-1}_\de(\gamma_{\de,\beta})$ (close to $0$).

The curve $\gamma_{\de,0}(t)$ intersect circle $S(0,|\de|^h)$ for $t=h\log_{|\la_\de|}|\de|$. Let us denote the intersection point by $b_{\de,h}$. Then, $\arg (b_{\de,h})=h\log_{|\la_\de|}|\de|\cdot\arg\la_\de$. Since $|\arg\la_\de|<|\de|$, we see that
  $$\arg (b_{\de,h})=h\log_{|\la_\de|}|\de|\cdot\arg\la_\de\rightarrow0,$$
uniformly, when $\de\rightarrow0$ and, $2\greq h\greq3/10$. Moreover
  $$\dist (b_{\de,h},\R^+)\leeq |\de|^h\cdot h\log_{|\la_\de|}|\de|\cdot\arg\la_\de,$$
thus for $ h\greq3/10$ we obtain
  $$\dist (b_{\de,h},\R^+)< |\de|^{5/4},$$
where $0<|\de|<\eta$, for suitably chosen $\eta>0$.

Next, if a line $e^{i\theta}\R^+$ intersects $\gamma_{\de,0}$ at a point $b_{\de,h'}$ where $2\greq h'\greq3/10$, we also have
\begin{equation}\label{eq:sd}
   \dist (b_{\de,h},e^{i\theta}\R^+)< |\de|^{5/4},
\end{equation}
for $h\greq h'$. Moreover this estimate will holds after rotation, so we can allow $\beta\neq0$.

%The same estimate holds for $\gamma_{\de,\beta}$ and  the line $e^{i\beta}\cdot\R^+$.

Let $\beta_\de$ and $t_\de$ be such that $\gamma_{\de,\beta_\de}(t_\de)=-|\de|^{5/16}+3i|\de|^{1/2}$ (then $\gamma_{\de,\beta_\de}(t_\de)$ lies on a circle $S(0,|\de|^{h_\de})$, where $h_\de$ is close to $5/16>3/10$). Because $\hat \Phi^{-1}_\de=z+a_\de z^2+O(z^3)$, we see that
   $$\im(\hat \Phi^{-1}_\de(\gamma_{\de,\beta_\de}(t_\de)))>2|\de|^{1/2}.$$
The same estimate is also valid for parameters close to $t_\de$, therefore we conclude from Lemma \ref{lem:sqrt|de|} that the Julia set lies below the curve $\hat \Phi^{-1}_\de(\gamma_{\de,\beta_\de})$.

Let $z=x+iy$, then $l_\de:y=-3|\de|^{3/16}x$, is the line that intersects $\gamma_{\de,\beta_\de}(t)$ at $t_\de$. We have $l_\de(-|\de|^{15/16})=3|\de|^{9/8}$. Let $t_\de'$ be such that $\re(\gamma_{\de,\beta_\de}(t_\de'))=-|\de|^{15/16}$. Then, $|\gamma_{\de,\beta_\de}(t_\de')-3|\de|^{9/8}|$ is close to $\dist(\gamma_{\de,\beta_\de}(t_\de'),l_\de)$, hence (\ref{eq:sd}) gives us
   $$|\gamma_{\de,\beta_\de}(t_\de')-3|\de|^{9/8}|<2|\de|^{5/4}.$$
Since $\hat \Phi^{-1}_\de=z+a_\de z^2+O(z^3)$, we see that
   $$\im(\hat \Phi^{-1}_\de(\gamma_{\de,\beta_\de}(t_\de')))<|\de|^{17/16}.$$
The same estimate are also valid for parameters less than $t_\de'$. Analogously, considering a curve $\gamma_{\de,\beta_\de^-}$ passing through the point $-|\de|^{5/16}-3i|\de|^{1/2}$, we can get lower estimate. So, the first statement follows from the symmetry.

The second statement follows from the fact that $\hat{\mc J_\de}$ lies between the curves $\hat \Phi^{-1}_\de(\gamma_{\de,\beta_\de^{\pm}})$, whereas arguments of the intersection points with $S(0,|\de|^2)$ are close to $\pi$.
\end{proof}

\subsection{The uniform convergence}
Now we are going to prove that $\vp_\de$ converges uniformly to $\vp_0$.

Let $ U_\de$, where $\de\in\mc  U$, be a simple connected neighborhood of $\mc J_{\de}$ containing $0$, disjoint from the postcritical set $P(f_\de)$, where
$$P(f_\de):=\overline{\bigcup_{n\greq1}f_\de^n(0)}.$$
If $\de=0$, then we assume that $(-2,2)\subset U_0$, and $U_0$ does not contain $\{-2,2\}$.

Let $f_{\de_0,\nu}^{-n}$ be an inverse branch of $f_{\de_0}^n$ defined on $ U_{\de_0}$. Then we denote by $f_{\de,\nu}^{-n}$ related inverse branches of $f_\de^n$ defined on $U_\de$, where $\de\in\mc U\cup\{0\}$ (by related we mean that the function $\de\mapsto f_{\de,\nu}^{-n}(0)$ is continuous). Let $V_n$ be the set of all possible choices of the inverse branches. For $\nu\in V_n$, $\de\in\mc U$ we write
$$\mc C_{\de,\nu}:=f_{\de,\nu}^{-n}(\mc J_\de).$$
If $\de=0$, then we take $\mc C_{0,\nu}:=\overline{f_{0,\nu}^{-n}(-2,2)}$.

\begin{lem}\label{lem:Cnu}
For every $\alpha\in(0,2\pi)$ and $n\greq1$ we have
$$\lim_{\de\rightarrow0}\sup_{\nu\in V_n}d_H(\mc C_{\de,\nu},\mc C_{0,\nu})=0,$$
where $\alpha=\arg\de$.
\end{lem}

\begin{proof}
Fix $n\greq1$. If $\arg\de=\pi$, then every two consecutive cylinders (that are included in $\R$) are separated by an appropriate preimage of $0$. Because $f_{\de,\nu}^{-n}(0)\rightarrow f_{0,\nu}^{-n}(0)$ and $p_\de\rightarrow p_0$, the assertion follows.
Of course cylinders are also separated by preimages of the imaginary axis $Y$, that are dinjoint from $\mc J_\de$.

Fix $\alpha\in(0,2\pi)\sms\{\pi\}$. If we prove that $f_{\de,\nu}^{-n}(Y)$ are disjoint from $\mc J_\de$ also for $\arg\de=\alpha$, the assertion will follow from Corollary \ref{cor:dH}.

We have $f_\de(Y)=-2+\de+\R^-$, so the imaginary part of every point from $f_\de(Y)$ is equal to $\im\de$, whereas $\im(- p_{\de})\approx\im\de/3$. Thus, we conclude from Lemma \ref{lem:|de|}, that $f_\de(Y)$ is disjoint from the Julia set, so also $f_{\de,\nu}^{-n}(Y)$.
\end{proof}

Since we have
$$\lim_{n\rightarrow\infty}\sup_{\nu\in V_n}\diam(\mc C_{0,\nu})=0,$$
Lemma \ref{lem:Cnu} leads to:

\begin{prop}\label{prop:uniform}
If $\alpha\in(0,2\pi)$ and $\de\rightarrow0$ where $\alpha=\arg\de$, then the convergence
$$\vp_\de\rightrightarrows\vp_0,$$
is uniform on the set $\mc J_{\de_0}$.
\end{prop}

%$f_{\de,+}^{-1}:\mc J_{\de}\rightarrow\mc J_{\de}\cap\{z\in\C:\re z>0\}$, $f_{\de,+}^{-n}:=(f_{\de,+}^{-1})^n$

%$f_{\de,-}^{-1}:\mc J_{\de}\rightarrow\mc J_{\de}\cap\{z\in\C:\re z<0\}$, $f_{\de,-}^{-n}:=(f_{\de,-}^{-1})^n$

\section{Postcritical set}\label{sec:post}

%$$P(f_\de):=\overline{\bigcup_{n=1}^{\infty}f_\de^n(0)}$$

%$$E_r:\frac{x^2}{(r+1/r)^2}+\frac{y^2}{(r-1/r)^2}\leeq1$$
%Focuses are $\pm2$.

Now we give a Proposition that helps us to control the postcritical set $P(f_\de)$. In particular we will see that $P(f_\de)$ is not to close to $0$.

\begin{prop}\label{prop:P}
For every $\alpha\in(0,2\pi)$ and $\theta>0$, there exist $r>0$ and $\eta>0$ such that
   $$P(f_\de)\subset\big((\C\sms E_r)\cup B(-2,\theta)\cup B(2,\theta)\big),$$
where $0<|\de|<\eta$ and $\alpha=\arg\de$.
\end{prop}

\begin{proof}
Fix $\alpha\in(0,2\pi)$ and $\theta>0$. Let $\de_1$, where $\arg\de_1=\alpha$, be such that
\begin{equation}\label{eq:3/8}
\big|\de_1\big|<\frac38\min\Big(r_{\textrm{z}},\frac\theta2\Big),
\end{equation}
and
\begin{equation}\label{eq:dist}
\frac34\dist\Big(-\frac83t\cdot \de_1,\R^-\Big)<\dist\Big(\hat\Phi^{-1}_{0}\Big(-\frac83t\cdot\de_1\Big),\R^-\Big),
\end{equation}
for $t\in(0,1]$.

Write $\de_t:=t\cdot\de_1$.
We will consider sequences of parameters of the form $\de_t/4^n$, for $t\in(1/4,1]$.

Let $v_{t,n}:=f^2_{\delta_t/4^n}(0)-p_{\delta_t/4^n}$. Then, using (\ref{eq:Phi-1}) we obtain
\begin{multline}\label{eq:bbb}
f_{\delta_t/4^n}^{n+2}(0)=f_{\delta_t/4^n}^n(p_{\delta_t/4^n}+v_{t,n})\\= f_{\delta_t/4^n}^n\Big( p_{\delta_t/4^n}+\frac{v_{t,n}\lambda_{\delta_t/4^n}^n}{\lambda_{\delta_t/4^n}^n}\Big) =\hat\Phi_{n,\delta_t/4^n}^{-1}(v_{t,n}\lambda_{\delta_t/4^n}^n)+p_{\delta_t/4^n},
\end{multline}
provided $\hat\Phi_{n,\delta_t/4^n}^{-1}(v_{t,n}\lambda_{\delta_t/4^n}^n)$ is well defined, that is  $|v_{t,n}\lambda_{\delta_t/4^n}^n|<r_{\textrm{z}}$, for $n$ large enough.

Now we will compute the limit of $v_{t,n}\lambda_{\delta_t/4^n}^n$ when $n\rightarrow\infty$. First, note that
$$p_{\delta_t/4^n}+v_{t,n}=f^2_{\delta_t/4^n}(0)=2-3\frac{\delta_t}{4^n}+\frac{\delta_t^2}{4^{2n}}.$$
Because
$$p_{\delta_t/4^n}=2-\frac13\frac{\delta_t}{4^n}+O\Big(\frac{\delta_t^2}{4^{2n}}\Big),$$
we get
$$v_{t,n}=-\frac83\frac{\delta_t}{4^n}+O\Big(\frac{\delta_t^2}{4^{2n}}\Big).$$
Therefore
\begin{equation}\label{eq:bb}
   v_{t,n}\lambda_{\delta_t/4^n}^n= -\frac83\frac{\lambda_{\delta_t/4^n}^n}{4^n}\de_t+O\Big(\frac{\lambda_{\delta_t/4^n}^n}{4^{2n}}\delta_t^2\Big).
\end{equation}

We have
$$\lambda_\delta=f'_\delta(p_\delta)=4+3\Big(\sqrt{1-\frac49\delta}-1\Big),$$
hence
$$\frac{\lambda_{\delta_t/4^n}^n}{4^n}=\bigg(1+\frac34\Big(\sqrt{1-\frac49\frac{\delta_t}{4^n}}-1\Big)\bigg)^n.$$
Since $(\sqrt{1-\frac49\frac{\delta_t}{4^n}}-1)\rightarrow0$, and
$$\Big(\sqrt{1-\frac49\frac{\delta_t}{4^n}}-1\Big)n\rightarrow0,$$
we obtain
$$\frac{\lambda_{\delta_t/4^n}^n}{4^n}= \bigg(1+\frac34\Big(\sqrt{1-\frac49\frac{\delta_t}{4^n}}-1\Big)\bigg)^n\rightarrow e^0=1.$$
Combining this with (\ref{eq:bb}), we see that
$$v_{t,n}\lambda_{\delta_t/4^n}^n\rightarrow-\frac83\delta_t.$$

So, for $n$ large enough, the assumption (\ref{eq:3/8}) leads to $|v_{t,n}\lambda_{\delta_t/4^n}^n|<r_{\textrm{z}}$.
Therefore $\hat\Phi_{n,\delta_t/4^n}^{-1}(v_{t,n}\lambda_{\delta_t/4^n}^n)$ is well defined, and using (\ref{eq:bbb}) we obtain
\begin{equation}\label{eq:8/3}
f_{\delta_t/4^n}^{n+2}(0)=\hat\Phi_{n,\delta_t/4^n}^{-1}(v_{t,n}\lambda_{\delta_t/4^n}^n)+p_{\delta_t/4^n}\rightrightarrows \hat\Phi_0^{-1}\Big(-\frac83\delta_t\Big)+2,
\end{equation}
where convergence is uniform with respect to $t\in(1/4,1]$. Next, the assumption (\ref{eq:dist}) gives us
\begin {equation*}\label{eq:dis2}
\frac12\dist\Big(-\frac83 \de_t+2,[-2,2]\Big)<\dist\Big(f_{\delta_t/4^n}^{n+2}(0),[-2,2]\Big),
\end{equation*}
for $t\in(1/4,1]$ and $n>n_0$ for suitable $n_0\in\N$. So, there exists $r>1$ such that
\begin{equation}\label{eq:22}
f_{\delta_t/4^n}^{n+2}(0)\in\C\sms E_r,
\end{equation}
for $t\in(1/4,1]$ and $n>n_0$. Since $f_0(E_r)=E_{r^2}$ and distance between $E_r$ and $E_{r^2}$ is equal to $(r^2+1/r^2)-(r+1/r)$ (semi-major axis length difference), we conclude that
\begin{equation*}
f_{\delta}(\C\sms E_r)\subset\C\sms E_r,
\end{equation*}
provided $|\de|<(r^2+1/r^2)-(r+1/r)$. Changing $n_0$ if necessary, we can assume that $\de_1/4^{n_0}<(r^2+1/r^2)-(r+1/r)$ and then (\ref{eq:22}) gives us
\begin{equation*}
f_{\delta_t/4^n}^{n+2+k}(0)\in\C\sms E_r,
\end{equation*}
for $t\in(1/4,1]$, $n>n_0$ and $k\greq0$.

Using (\ref{eq:8/3}), (\ref{eq:PP}) and the assumption (\ref{eq:3/8}), we obtain $f_{\delta_t/4^n}^{n+2}(0)\in B(2,\theta)$. Since $f_{\delta_t/4^n}(0)\in B(-2,\theta)$ and then $f_{\delta_t/4^n}^{k+2}(0)\in B(2,\theta)$, where $0\leeq k< n$, $n>n_0$ and $t\in(1/4,1]$, the statement holds for $\eta=|\de_1|/4^{n_0}$.
\end{proof}

\section{Cylinders}\label{sec:cyl}

In this section we define cylinders $\mc C_{\de,n}^{-2}$, $\mc C_{\de,n}^{0}$, $\mc C_{\de,n}^{+2}$, that give partitions of neighborhoods of the points $-p_\de$, $0$, $p_\de$, respectively. However, first we will deal with inverse branches of $f_\de$.

Since $\de_0\in\R^-$, the Julia set $\mc J_{\de_0}$ is disjoint with the imaginary axis, so let us write
$$
\begin{array}{lll}
f_{\de_0,+}^{-1}:\mc J_{\de_0}\rightarrow\mc J_{\de_0,+}=\mc J_{\de_0}\cap\{z\in\C:\re z>0\},\\
f_{\de_0,-}^{-1}:\mc J_{\de_0}\rightarrow\mc J_{\de_0,-}=\mc J_{\de_0}\cap\{z\in\C:\re z<0\}.
\end{array}
$$
Using the holomorphic motion $\vp_\de$, we can define $f_{\de,\pm}^{-1}$ for parameters $\de\in\mc U$, namely: $f_{\de,\pm}^{-1}:\mc J_{\de}\rightarrow\vp_\de(\mc J_{\de_0,\pm})$.

But, for $\de$ close to $0$ (i.e. $0<|\de|<\eta(\alpha)$), the Julia sets are also disjoint from $i\R$ (see proof of Lemma \ref{lem:Cnu}), so we can define $f_{\de,\pm}^{-1}$ analogously as for $\de_0$:
$$
\begin{array}{lll}
f_{\de,+}^{-1}:\mc J_{\de}\rightarrow\mc J_{\de,+}=\mc J_{\de}\cap\{z\in\C:\re z>0\},\\
f_{\de,-}^{-1}:\mc J_{\de}\rightarrow\mc J_{\de,-}=\mc J_{\de}\cap\{z\in\C:\re z<0\}.
\end{array}
$$
For $\de=0$ we take $f_{0,+}^{-1}:[-2,2]\rightarrow[0,2]$, and $f_{0,-}^{-1}:[-2,2]\rightarrow[-2,0]$.
Moreover
$$f_{\de,+}^{-n}:=(f_{\de,+}^{-1})^n, \quad f_{\de,-}^{-n}:=(f_{\de,-}^{-1})^n.$$

\begin{figure}[!ht]
   \centering
        \includegraphics[scale=0.72]{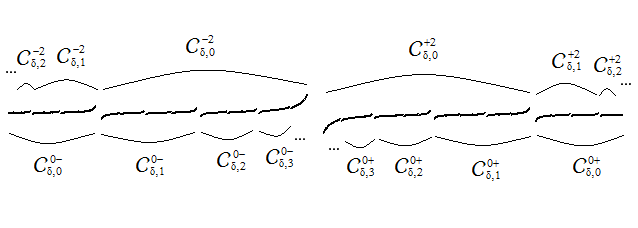}
        \caption{Partition $\mc J_{\de}$ onto cylinders, where $\de=\frac{1}{50}+\frac{1}{50}i$.}
\end{figure}

Let us now define partition of $\mc J_{\de,+}$ onto cylinders $\mc C_{\de,n}^{+2}$, $n\greq0$, which can be used to describe neighborhoods of $p_\de\approx2$.

First, note that $\mc J_{\de,+}=\mc J_{\de,+-}\cup\mc J_{\de,++}$, where $\mc J_{\de,+-}:=f_{\de,+}^{-1}(\mc J_{\de,-})$, and $\mc J_{\de,++}:=f_{\de,+}^{-1}(\mc J_{\de,+})$. We take
$$\mc C_{\de,0}^{+2}:=\mc J_{\de,+-},\quad\textrm{ and }\quad\mc C_{\de,n}^{+2}:=f_{\de,+}^{-n}(\mc C_{\de,0}^{+2}).$$
In particular we have $\mc J_{\de,++}=\bigcup_{n=1}^{\infty}\mc C_{\de,n}^{+2}\cup\{p_\de\}$ and
$$\mc J_{\de,+}=\bigcup_{n=0}^{\infty}\mc C_{\de,n}^{+2}\cup\{p_\de\}.$$

Cylinders $\mc C_{\de,n}^{-2}$, $n\greq0$, describing neighborhoods of $-p_\de\approx-2$, are placed symmetrically with respect to $0$. So we obtain$$\mc J_{\de,-}=\bigcup_{n=0}^{\infty}\mc C_{\de,n}^{-2}\cup\{-p_\de\}.$$

In order to describe partition of $\mc J_\de$ close to $0$, we take
$$\mc C_{\de,0}^{0+}:=\mc J_{\de,++},\quad\textrm{ and }\quad\mc C_{\de,n}^{0+}:=f_{\de,+}^{-1}(\mc C_{\de,n-1}^{-2})\quad \textrm{for }n\greq1,$$
thus $\mc J_{\de,+}=\bigcup_{n=0}^{\infty}\mc C_{\de,n}^{0+}\cup\{f_{\de,+}^{-1}(-p_\de)\}$, whereas cylinders $\mc C_{\de,n}^{0-}$, $n\greq0$ are placed symmetrically with respect to $0$. Finally
$$\mc C_{\de,n}^{0}:=\mc C_{\de,n}^{0-}\cup\mc C_{\de,n}^{0+}.$$
Note that
$$\mc J_\de=\bigcup_{n=0}^{\infty}\mc C^{0}_{\de,n}\cup f_{\de}^{-1}(\{-p_\de\}).$$

We have the following important relations
$$f_\de(\mc C^{0+}_{\de,n})=f_\de( \mc C^{0-}_{\de,n})=\mc C^{-2}_{\de,n-1},\qquad f_\de(\mc C^{-2}_{\de,n})= \mc C^{+2}_{\de,n-1},$$
for $n\greq1$, hence $f_\de^2(\mc C^{0+}_{\de,n})=f_\de^2( \mc C^{0-}_{\de,n})=\mc C^{+2}_{\de,n-2}$, for $n\greq2$.

We will also consider the sets:
$$\mc M_{\de,N}^{0}=\bigcup_{n\greq N}\mc C^{0}_{\de,n}.$$
Next, we take $C_n^0:=\vp_\de^{-1}(\mc C_{\de,n}^0)$ and $M_N^0:=\vp_\de^{-1}(\mc M_{\de,N}^0)$ (subsets of $\mc J_{\de_0}$).

Notation: Let $\mc C_{\de,n}^{\pm2}:=\mc C_{\de,n}^{-2}\cup\mc C_{\de,n}^{+2}$, whereas $\mc C_{\de,n}^{0\pm}$ means "$\mc C_{\de,n}^{0-}$ or $\mc C_{\de,n}^{0+}$".

\begin{lem}\label{lem:mm}
   For every $\alpha\in(0,2\pi)$ there exist $K>1$ and $\eta>0$ such that if $z\in\mc C^{0}_{\de,n}$, $n\greq0$, then
     \begin{enumerate}

       \item\label{lit:mm1}
        $|z|>K^{-1}|\lambda_\de|^{-n/2} $,
        \item\label{lit:mm2}
        $|z|<K|\lambda_\de|^{-n/2} $, provided $|z|> \sqrt{|\de|}$,
       \item\label{lit:mm3}
         $\diam\mc C_{\de,n}^{0\pm} < K|\lambda_\de|^{-n/2}$,
     \end{enumerate}
    where $0<|\delta|<\eta$ and $\alpha=\arg\delta$.
\end{lem}

\begin{proof}
Fix $\alpha\in(0,2\pi)$ and let $\alpha=\arg \de$.
Let $n_0\in\N$ be the smallest number for which $f_0^2(\mc C^{0}_{0,n_0})\subset B(2,r_\textrm{z})$. Then, there exists $\eta>0$, such that $f_\de^2(\mc C^{0}_{\de,n_0})\subset B(2,r_\textrm{z})$, for $0<|\de|<\eta$.

\emph{Step 1.}
If $z\in\mc C^{0}_{\de,n}$, where $n\greq n_0$, then
   $$|f_\de(z)-f_\de(0)|=|z|^2.$$
Because $-p_\de\approx-2+\de/3$ (cf. (\ref{eq:1/27})), whereas $f_\de(0)=-2+\de$, using Lemma \ref{lem:|de|} we get
   $$ K_1^{-1} |f_\de(z)-(-p_\de)|<|f_\de(z)-f_\de(0)|,$$
for some $K_1>0$ (depending on $\alpha$). If $|z|>\sqrt{|\de|}$ then $f_\de(z)\notin B(-2+\de,|\de|)$, so changing $K_1$ if necessary, we obtain
   $$ |f_\de(z)-f_\de(0)|<K_1 |f_\de(z)-(-p_\de)|.$$
Therefore
\begin{equation}\label{eq:z}
   K_1^{-1}|f_\de(z)-(-p_\de)|<|z|^2< K_1|f_\de(z)-(-p_\de)|,
\end{equation}
where right-hand side inequality holds under the assumption $|z|>\sqrt{|\de|}$.

We have $f^{n-n_0+1}_\de(f_\de(z))\in f_\de^2(\mc C^{0}_{\de,n_0})$ and of course $f_\de^{n-n_0+1}(-p_\de)=p_\de$.
Since $f_\de$ is conjugated to $z\mapsto \la_\de z$ on $B(2,r_\textrm{z})$, we get
   $$|f^{n-n_0+2}_\de(z)-p_\de|\asymp |\lambda_\de|^{n-n_0+1}|f_\de(z)-(-p_\de)|.$$
The fact that $|f^{n-n_0+2}_\de(z)-p_\de|$ is separated from zero, leads to
   $$|\lambda_\de|^{n-n_0+1}|f_\de(z)-(-p_\de)|\asymp1, \quad\textrm{ hence }\quad |f_\de(z)-(-p_\de)|\asymp|\lambda_\de|^{-n}.$$
Thus, the first two statements for $n\greq n_0$ follow from (\ref{eq:z}).

\emph{Step 2.}
If $z\in\mc C^{0}_{\de,n}$, $n\greq n_0$ then the first statement gives us $|f_\de'(z)|>K_2^{-1}|\la_\de|^{-n/2}$, so we conclude that
   $$\diam f_\de(\mc C^{0}_{\de,n})=\diam f_\de(\mc C_{\de,n}^{0\pm})>K_3^{-1}|\la_\de|^{-n/2}\diam\mc C_{\de,n}^{0\pm}.$$
Since $f_\de^{n-n_0+1}(f_\de(\mc C^{0}_{\de,n}))= f_\de^2(\mc C^{0}_{\de,n_0})$, we obtain
   $$|\la_\de|^{n-n_0+1}\diam f_\de(\mc C^{0}_{\de,n})\asymp\diam f_\de^2(\mc C^{0}_{\de,n_0}).$$
Therefore
   $$\diam \mc C_{\de,n}^{0\pm}<K_4|\la_\de|^{-n/2+n_0-1}\diam f_\de^2(\mc C^{0}_{\de,n_0})<K_5|\la_\de|^{-n/2}, $$
so all three statements hold for $n\greq n_0$, and some constant $K>0$.

Next, possibly changing $K$, we can assume that the statements hold for $n\greq0$.
\end{proof}

\section{The Julia set close to 0}\label{sec:close}

In order to compute the integral from the numerator of the formula (\ref{eq:d}), we need precise information on the location of the Julia set close to $0$. Corollary \ref{cor:dH} tells us nothing precise for $\de\neq0$, so we need to change scale in order to see the details. Because preimages of the point $-p_\de$ are close to $\pm\frac{\sqrt6}{3}\sqrt{-\de}$, we change scale by the factor $1/\sqrt{|\de|}$ in every direction. Thus, let us write
$$\mc J^\divideontimes_\delta:=\frac{1}{\sqrt
{|\delta|}}\mc J_\delta.$$
Next, we define small neighborhoods of $0$, and its rescaled version, as follows:
\begin{equation*}
\begin{array}{ll}
      \mc N_{\delta,R}:=\mc J_\delta\cap \{z\in\C:|\re z|\leeq R\sqrt{|\de|}\}, \\
      \mc N^\divideontimes_{\delta,R}:=\mc J^\divideontimes_\delta\cap \{z\in\C:|\re z|\leeq R\},
\end{array}
\end{equation*}
where $R\greq1$. So, we have $\mc N_{\delta,R}=\sqrt{|\de|}\mc N^\divideontimes_{\delta,R}$. Write $N_R:=\vp_\de^{-1}(\mc N_{\delta,R})\subset \mc J_{\de_0}$.

We denote by $b_\de$ one of the points being close to a preimage of $-p_\de$, namely
$$ b_\de:=\frac{\sqrt6}{3}\sqrt{-\de}=\frac{\sqrt6}{3}\sqrt{|\de|}\sin\frac\alpha2-i\frac{\sqrt6}{3}\sqrt{|\de|}\cos\frac\alpha2.$$
Indeed $f_\de(b_\de)=-2+\frac13\,\de\approx-p_\de$ (cf. (\ref{eq:1/27})).

We need to define the following sets:
$$H_{\de,R}:\left\{ \begin{array}{ll}
xy=-\frac13\im\delta=-\frac13|\de|\sin\alpha,\\
\frac{\sqrt6}{3}\sqrt{|\de|}\,\sin\frac\alpha2\leeq|x|\leeq R\sqrt{|\de|},
\end{array} \right.$$
where $x=\re z$, $y=\im z$, $\de\neq0$, $\arg\de=\alpha\in(0,2\pi)$, and $R\greq1$.
Branches of $H_{\de,R}$ included in left and right half-planes will be denoted by $H^-_{\de,R}$ and $H^+_{\de,R}$ respectively. The endpoints of $H^+_\de$ are $b_\de$, and the point which will be denoted by $z_{\de,R}$ (i.e. $\re z_{\de,R}=R\sqrt{|\de|}$).

After substitution $\sqrt{|\de|}x$, $\sqrt{|\de|}y$ in place of $x$, $y$, for $\alpha\in(0,2\pi)$ we obtain
$$H^\divideontimes_{\alpha,R}:\left\{ \begin{array}{lll}
xy=-\frac13\sin\alpha,\\
\frac{\sqrt{6}}{3}\,\sin\frac\alpha2\leeq |x|\leeq R.
\end{array} \right.$$
Thus we have $H_{\de,R}=\sqrt{|\de|}H^\divideontimes_{\alpha,R}$, where $\alpha=\arg\de$. We define $H^{\divideontimes,-}_{\alpha,R}$ and $H^{\divideontimes,+}_{\alpha,R}$ analogously as $H_{\de,R}^\pm$. The endpoints of $H^{\divideontimes,+}_{\alpha,R}$ are
\begin{equation}\label{eq:endp}
   b^{\divideontimes}_\alpha:=\frac{1}{\sqrt{|\de|}}\,b_\de=\frac{\sqrt6}{3}\sin\frac\alpha2-i\frac{\sqrt6}{3}\cos\frac\alpha2, \quad z_{\alpha,R}^\dt:=R-i\frac{1}{3R}\sin\alpha.
\end{equation}
Note that $\arg (z_{\alpha,R}^\dt)=\arg (z_{\de,R})=-\gamma_{\alpha,R}$, where
\begin{equation}\label{eq:gamma}
\gamma_{\alpha,R}=\arctan\Big(\frac{1}{3R^2}\sin\alpha\Big).
\end{equation}
Hence, $\gamma_{\alpha,R}>0$ for $\alpha\in(0,\pi)$.

\begin{figure}[!ht]
   \centering
        \includegraphics[scale=0.65]{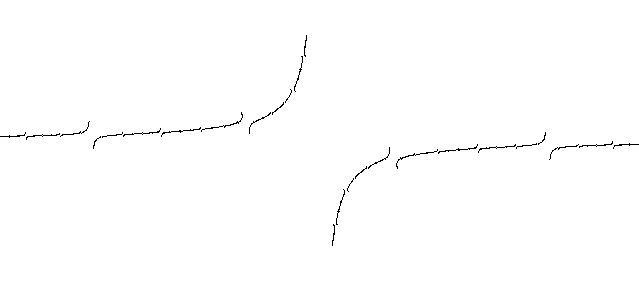}
        \caption{The Julia set $\mc J_\de$ for $\de=\frac{1}{25}+\frac{1}{100}i$, close to $0$.}
\end{figure}

%Later on we will also consider measures defined on $\mc J^\divideontimes_\delta$ and $\mc N^\divideontimes_{\delta,R}$. Thus let us define
%$$\omega^\divideontimes_\delta(A):=\frac{1}{\sqrt
%{|\delta|}}\,\omega_\de(\sqrt{|\de|}A),\quad \mu^\divideontimes_\delta(A):=\frac{1}{\sqrt
%{|\delta|}}\,\mu_\de(\sqrt{|\de|}A),$$
%measures on the set $\mc J^\divideontimes_\delta$, and
%$$\omega^\divideontimes_\delta\big|_{\mc J^\divideontimes_{\delta,R}}=:\omega^\divideontimes_{\delta,R},\quad\quad \mu^\divideontimes_\delta\big|_{\mc J^\divideontimes_{\delta,R}}=:\mu^\divideontimes_{\delta,R},$$
%measures on the set $\mc N^\divideontimes_{\delta,R}$.

Now we prove the following:

\begin{prop}\label{prop:Hmetric}
   For every $\alpha\in(0,2\pi)$ and $R\greq1$, if $\delta\rightarrow0$ where $\alpha=\arg\delta$, then
      $$\mc N^\divideontimes_{\delta,R}\rightarrow H^\divideontimes_{\alpha,R},$$
   in the space of non-empty compact subsets of $\C$ equipped with the Hausdorff metric.
\end{prop}

\begin{proof}
Fix $\alpha\in(0,2\pi)$ and $R\greq1$.

\emph{Step 1.}
First we will prove that
\begin{equation}\label{eq:d1}
    \lim_{\de\rightarrow0}\sup_{z\in\mc N^\divideontimes_{\delta,R}}\dist(z,H^\divideontimes_{\alpha,R})=0,
\end{equation}
where $\alpha=\arg\de$.

If $z=x+iy$, then we define
$$l_\de^{+}:xy=-\frac13\im\de+|\de|^{17/16},\quad l_\de^{-}:xy=-\frac13\im\de-|\de|^{17/16}.$$
The images of $l_\de^{\pm}$ are horizontal lines such that
$$\im(f_\de(l_\de^\pm)\big)=\frac13\im\de\pm2|\de|^{17/16}.$$

The set $|\re z|\leeq R\sqrt{|\de|}$ is mapped into the half-plane $\re z\leeq-2+R^2|\de|+\re\de$. Since $-p_\de=-2+\de/3+O(\de^2)$ (cf, (\ref{eq:1/27})), we conclude from Lemma \ref{lem:|de|} (\ref{lit:|de|1}) that the set $\mc N_{\de,R}$ lies between the lines $l_\de^\pm$.

Moreover $\mc N_{\de,R}$ is bounded by $f_{\de}^{-1}(k_\de)$ where $k_\de:\re z=-\re p_\de-|\de|^2$ (see Lemma \ref{lem:|de|} (\ref{lit:|de|2})). Since $f_\de(b_\de)=-2+\de/3$, (\ref{eq:1/27}) leads to $|-p_\de-f_\de(b_\de)|\approx|\de|^2/27$, whereas the critical value is equal to $-2+\de$. Thus we see that the distance between $f_{\de}^{-1}(k_\de)$ and the points $\pm b_\de$ is of order $|\de|^{3/2}$.

So, the set $\mc N_{\de,R}^\divideontimes$ lies between the curves
$$xy=-\frac13\sin\alpha\pm|\de|^{1/16},$$
and rescaled curves $f_\de^{-1}(k_\de)$ whose distance from $\pm b^\divideontimes_\alpha$ is of order $|\de|$. Thus (\ref{eq:d1}) holds.

\emph{Step 2.}
Now we will prove that
\begin{equation}\label{eq:d2}
    \lim_{\de\rightarrow0}\sup_{z\in H^\divideontimes_{\alpha,R}}\dist(z,\mc N^\divideontimes_{\delta,R})=0.
\end{equation}
Suppose that, on the contrary, there is $z_0\in H^\divideontimes_{\alpha,R}$, $r_0>0$, and a sequence $\de_n\rightarrow0$ such that $B(z_0,r)\cap \mc N^\divideontimes_{\delta_n,R}=\emptyset$. Note that the distances between $\pm b^\divideontimes_\alpha$ and rescaled preimages of $f_\de^{-1}(-p_\de)$ are of order $|\de|$, so the rescaled Julia set is located on ''both sides'' of $B(z_0,r_0)$. We can assume that $0\notin B(z_0,r_0)$.

Next, we have $B(\sqrt{|\de_n|}z_0,\sqrt{|\de_n|}r_0)\cap \mc N_{\delta_n,R}=\emptyset$, and then
$$ f_{\de_n}\big(B(\sqrt{|\de_n|}z_0,\sqrt{|\de_n|}r_0)\big)\cap \mc J_{\delta_n}=\emptyset.$$
Because $f_{\de_n}'(\sqrt{|\de_n|}z_0)=\kappa\sqrt{|\de_n|}$, for some $\kappa\in\C\sms\{0\}$, Koebe one-quarter Theorem leads to
   $$B\big(f_{\de_n}(\sqrt{|\de_n|}z_0),\kappa r_0|\de_n|/4\big)\subset f_{\de_n}\big(B(\sqrt{|\de_n|}z_0,\sqrt{|\de_n|}r_0)\big).$$
Since $\im f_{\de_n}(\sqrt{|\de_n|}z_0)=\im \de_n/3$, the point $f_\de(\sqrt{|\de_n|}z_0)$ lies between the lines $\im z=-\im p_{\de_n}\pm|\de_n|^{17/16}$. So, we conclude from Lemma \ref{lem:|de|} (\ref{lit:|de|1}) that there exists $0<r_1\leeq\kappa r_0/4$, such that
   $$A_{\alpha,n}:=\{z\in\C:|\re (z-f_{\de_n}(\sqrt{|\de_n|}z_0))|\leeq r_1|\de_n|\}\cap \mc J_{\delta_n}=\emptyset.$$
The diameter of the set $f_{\de_n}(\mc N_{\de_n,R})$ is close to $R^2|\de_n|$, whereas  $f_{\de_n}(\mc N_{\de_n,R})$ can be mapped with bounded distortion onto a set of diameter grater than $r_{\textrm{z}}/5$. So, the ''width'' of the images of $A_{\alpha,n}$ is separated from $0$. This contradicts Corollary \ref{cor:dH}. Therefore (\ref{eq:d2}) holds, and the statement follows.
\end{proof}

\section{Conformal and Invariant measures}\label{sec:measures}

Recall from Section \ref{sec:formalism}, that $\omega_\de$ (after normalization) is the conformal measure with exponent $d(\de)$ (cf. (\ref{eq:SP})), whereas $\mu_\de$ is the $f_\de$-invariant measure, equivalent to $\omega_\de$. We assume that $\mu_\delta(\mc J_\de)=\omega_\delta(\mc J_\de)=4$.

Let us define the rescaled measures as follows:
$$\omega^\divideontimes_\delta(A):=\frac{1}{{\sqrt
{|\delta|}}^{d(\de)}}\,\omega_\de(\sqrt{|\de|}A),\quad\textrm{and}\quad \mu^\divideontimes_\delta(A):=\frac{1}{{\sqrt
{|\delta|}}^{d(\de)}}\,\mu_\de(\sqrt{|\de|}A).$$
Hence, $\omega^\divideontimes_\delta$ and $\mu^\divideontimes_\delta$ are supported on $\mc J^\divideontimes_\delta$.

The Main result of this Section is Proposition \ref{prop:omega}, which concerns the limit of $\omega^\divideontimes_\delta$ (so also $\mu^\divideontimes_\delta$) when $\de\rightarrow0$.
But first, we will deal with the limits of $\omega_\de$ and $\mu_\de$.

For $\de=0$ situation is as follows: The measure $\omega_0$ is simply the Lebesgue measure $\la$. Next, we know that the map $\vp_v(z)=(\frac4\pi)\arcsin(\frac{z}{2})$ conjugates $f_0$ to $V$ (see Section \ref{sse:ce}). Because the Lebesgue measure is $V$-invariant, the measure defined by $\vp_v^{*}\la(A)=\la(\vp_v(A))$ is $f_0$-invariant with density
\begin{equation}\label{eq:m0o}
\frac{d\mu_0}{d\omega_0}(z)=\frac2\pi\frac{1}{\sqrt{1-(z/2)^2}},
\end{equation}
(cf. \cite[Chapter V]{MS}).

First fact follows from \cite[Section 5.1]{RL}:%\cite[Theorem 11.2]{Mii} that:

\begin{prop}\label{prop:o}
For every $\alpha\in(0,2\pi)$ the measure $\omega_0$ is equal to weak* limit of $\omega_\delta$, where $\delta\rightarrow0$ and $\alpha=\arg\delta$.
\end{prop}

Now we prove:

\begin{prop}\label{prop:mu}
For every $\alpha\in(0,2\pi)$ the measure $\mu_0$ is equal to weak* limit of $\mu_\delta$, where $\delta\rightarrow0$ and $\alpha=\arg\delta$.
\end{prop}

\begin{proof}
Fix $\alpha\in(0,2\pi)$.

\emph{Step 1.}
We see from \cite[Theorems 4.1, 4.12]{Z}) that the density $h_\de:=d\mu_\de/d\omega_\de$ is equal to the limit $\mathcal{L}^n_{\phi_\de}(\textbf{1})$, where $\phi_\de=-d(\de)\log|f'_\de(\varphi_\de)|$, $n\rightarrow\infty$. So, if $z=\vp_\de(s)$, then (\ref{eq:L}) leads to
\begin{equation}\label{eq:Ln}
   h_\de(z)=\lim_{n\rightarrow\infty}\mathcal{L}^n_{\phi_\de}(\textbf{1})(s)=\lim_{n\rightarrow\infty}\sum_{\overline s\in{T^{-n}(s)}}|(f_\de^n)'(\varphi_\de(\overline s))|^{-d(\de)}.
\end{equation}

Let $\de_{k}\rightarrow0$, where $\arg\de_{k}=\alpha$, be a sequence such that $\mu_{\de_k}$ tends to a measure $\hat\mu_0$ in the weak* topology.

If $z_0\in(-2,2)$, then we can find $r$ such that for $k$ large enough, and $\de$ close to $0$, $B(z_0,r)$ is separated from the postcritical set (cf. Proposition \ref{prop:P}). Then, all inverse branches of $f_{\de_k}^n$, $n\greq1$ have uniformly bounded distortion on $B(z_0,r)$. So, we see from (\ref{eq:Ln}) that there exists a constant $K_1>0$ such that
\begin{equation}\label{eq:Kk}
   h_{\de_k}\big|_{B(z_0,r)\cap\mc J_{\de_k}}<K_1.
\end{equation}
Passing to the limit (and possibly changing $K_1$), we obtain
\begin{equation}\label{eq:K}
    K^{-1}_1<h_{0}\big|_{(z_0-r,z_0+r)}<K_1,
\end{equation}
or $h_{0}=0$ on $B(z_0-r,z_0+r)$. In the latter case we get $h_{0}=0$ on $(-2,2)$.

The limit measure $\hat\mu_0$ is $f_0$-invariant. So, we see that if $\hat\mu_0$ has an atom, the only possibility is the fixed point $p_0=2$. Therefore, it is enough to show that there are no atom at $p_0$, because in that case $(\ref{eq:K})$ holds, and $\hat\mu_{0}$ is absolutely continuous with respect to $\omega_0$. Then, by uniqueness $\hat\mu_0=\mu_0$, so $\mu_0$ in the only possible limit, and the statement follows.

\emph{Step 2.}
So, now we prove that there are no atom at $p_0$.
If $\vp_\de(s)=z\in\mc C_{\de,n}^{+2}$, then using (\ref{eq:Ln}) we obtain
\begin{equation}\label{eq:h1}
   h_\de(z)=\sum_{k=0}^{\infty}|(f_{\de,-}^{-1}\circ f_{\de,+}^{-k})'(z)|^{d(\de)}\cdot h_\de((f_{\de,-}^{-1}\circ f_{\de,+}^{-k})(z)),
\end{equation}
where $(f_{\de,-}^{-1}\circ f_{\de,+}^{-k})(z)\in\mc C_{\de,n+k+1}^{-2}$. So, we must estimate values of $h_{\de}$ on the cylinders $\mc C_{\de,n+k+1}^{-2}$. We have
\begin{equation}\label{eq:hf}
   h_\de(z)=|(f_{\de,-}^{-1})'(z)|^{d(\de)}\cdot h_{\de}(f_{\de,-}^{-1}(z))+|(f_{\de,+}^{-1})'(z)|^{d(\de)} \cdot h_{\de}(f_{\de,+}^{-1}(z)).
\end{equation}
If $z\in \mc C_{\de,n+k+1}^{-2}$ then $f_{\de,\pm}^{-1}(z)\in \mc C_{\de,n+k+2}^{0\pm}$.

Let $z_0=0$ and let $r$ be such that (\ref{eq:Kk}) holds. Then, Lemma \ref{lem:mm} combined with Proposition \ref{prop:Hmetric}, implies that there exists $n_0\in\N$ such that $\mc\mc C_{\de,n+k+2}^{0\pm}\subset B(0,r)$ where $n\greq n_0$, $k\in\N$.

Again using Lemma \ref{lem:mm} (\ref{lit:mm1}), we get
$$|(f_{\de,\pm}^{-1})'(z)|^{d(\de)}=|f'_\de(f_{\de,\pm}^{-1}(z))|^{-d(\de)}<K_2|\la_\de|^{d(\de)(n+k+2)/2},$$
thus (\ref{eq:Kk}) and (\ref{eq:hf}) leads to
\begin{equation*}
    h_\de(z)<K_3|\la_\de|^{d(\de)(n+k+2)/2},
\end{equation*}
where $z\in \mc C_{\de,n+k+1}^{-2}$.

If $z\in \mc C_{\de,n}^{+2}$, then formula (\ref{eq:h1}) combined with the above estimate, and the fact that $|(f_{\de,-}^{-1}\circ f_{\de,+}^{-k})(z)|^{d(\de)}\asymp |\la_\de|^{-d(\de)(k+1)}$, gives us
\begin{equation*}
    h_\de(z)<  K_4\sum_{k=0}^{\infty}|\la_\de|^{d(\de)(n-k)/2}<K_5|\la_\de|^{d(\de)n/2},
\end{equation*}
where $n\greq n_0$.

We have $\omega_\de(\mc C_{\de,n}^{+2})\asymp(\diam \mc C_{\de,n}^{+2})^{d(\de)}\asymp |\la_\de|^{-d(\de)n}$. So, if $N\greq n_0$ then
\begin{multline*}
    \mu_\de\Big(\bigcup_{n=N}^{\infty} \mc C_{\de,n}^{+2}\Big)=\int_{\bigcup_{n=N}^{\infty} \mc C_{\de,n}^{+2}}h_\de \,d\omega_\de<K_6\sum_{n=N}^{\infty}|\la_\de|^{d(\de)n/2}|\la_\de|^{-d(\de)n}\\
    <K_6\sum_{n=N}^{\infty}4^{-n/3}<3K_6\Big(\frac14\Big)^{N/3}\xrightarrow[N\rightarrow\infty]{}0.
\end{multline*}
Because the above estimate does not depend on $\de$, the statement follows.
\end{proof}

\begin{lem}\label{lem:mu/omega}
For every $\alpha\in(0,2\pi)$ and $\ve>0$ there exist neighborhood $U\ni0$ and $\eta>0$ such that
$$\big(1-\ve\big)\frac2\pi<\frac{d\mu_\delta}{d\omega_\delta}\Big|_{U\cap\mc J_\de}<\big(1+\ve\big)\frac2\pi,$$
where $0<|\delta|<\eta$ and $\alpha=\arg\delta$ or $\de=0$.
\end{lem}

\begin{proof}
Fix $\alpha\in(0,2\pi)$ and $\ve>0$. As before, we write $h_\de=d\mu_\de/d\omega_\de$.

We see from (\ref{eq:m0o}) that there exists an interval $(-\kappa,\kappa)$ such that
\begin{equation}\label{eq:2pi1}
   \frac2\pi\leeq h_0(z)\leeq \Big(1+\frac\ve3\Big)\frac2\pi,
\end{equation}
where $z\in(-\kappa,\kappa)$. So the statement holds for $\de=0$.

%Now we prove that the densities $h_\de$ converges uniformly to $h_0$ on $[-\kappa,\kappa]$.

We can find (small) $0<r<\kappa$ such that distortion of all inverse branches of $f_\de^{n}$, $n\greq1$ on $B(0,r)$ is as close to $1$ as we need. So, taking into account formula (\ref{eq:Ln}), we can assume for every $z_1, z_2\in B(0,r)\cap\mc J_\de$, we have
\begin{equation}\label{eq:2pi2}
   \big|h_\de(z_1)-h_\de(z_2)\big|<\frac\ve3\frac2\pi.
\end{equation}

Since $h_0$ and $h_\de$ are close to constant functions on $(-r,r)$ and $B(0,r)\cap \mc J_\de$ respectively, we conclude from Propositions \ref{prop:o}, \ref{prop:mu} that there exist $z_3\in(-r,r)$ and $z_4\in B(0,r)\cap \mc J_\de$ such that
$|h_0(z_3)-h_\de(z_4)|<\frac\ve3\frac2\pi$. So the statement follows from (\ref{eq:2pi1}) and (\ref{eq:2pi2}).
\end{proof}

Let us denote by $l_{\alpha,R}$ the arc length measure supported on $H^\divideontimes_{\alpha,R}$. Now we prove the main result of this Section:

\begin{prop}\label{prop:omega}
For every $\alpha\in(0,2\pi)$ and $R\greq1$, we have
$$\omega^\divideontimes_{\delta}\Big|_{\mc N^\dt_{\de,R}}\rightarrow l_{\alpha,R}, $$
in the weak$^*$ topology, where $\delta\rightarrow0$ and $\alpha=\arg\delta$.% and $l_{\alpha,R}$ is the arc length measure on $H^\divideontimes_{\alpha,R}$.
\end{prop}

%\emph{Remark.
%In Section \ref{sec:int3}, at the end of the proof of Theorem \ref{thm:minus2}, we will see that $|\de|^{1-d(\de)}\rightarrow1$.}

\begin{proof}%[Proof of Proposition \ref{prop:omega}]
Fix $\alpha\in(0,2\pi)$, $R\greq1$, and let $\alpha=\arg\de$. It is enough to consider the set $\mc N_{\de,R}^{\dt,+}:=\mc N_{\de,R}^{\dt}\cap\{z\in\C:\re z>0\}$. Write $\mc N_{\de,R}^{+}=\sqrt{|\de|}\mc N_{\de,R}^{\dt,+}$.
%and $0<|\de|<r_\vartriangle$.

\emph{Step 1.}
Let $U'$ be a small neighborhood of $H^{\divideontimes,+}_{\alpha,R}$ separated from $0$. Let $U=U'\cap\{z\in\C:0<\re z<R\}$.
For any Borel set $A\subset U$, such that $f_\de^{n+1}$ is injective on $\sqrt{|\de|}A$, we have
\begin{multline}\label{eq:omegaint}
{\omega}^\divideontimes_{\delta}(A)=\frac{1}{{\sqrt{|\de|}}^{d(\de)}}\,{\omega}_{\delta}(\sqrt{|\de|}A) = \frac{1}{{\sqrt{|\de|}}^{d(\de)}}\int_{f_\de^{n+1}(\sqrt{|\de|}A)}\big|(f^{-n-1}_{\de,\nu})'\big|^{d(\de)}d\omega_\de\\
=\int_{f_\de^{n+1}(\sqrt{|\de|}A)} \bigg|\frac{(f^{-n-1}_{\de,\nu})'}{\sqrt{|\de|}}\bigg|^{d(\de)}d\omega_\de,
\end{multline}
where $f^{-n-1}_{\de,\nu}$ denotes the inverse branch that maps $f_\de^{n+1}(\sqrt{|\de|}U)$ onto $\sqrt{|\de|}U$.

Our aim is to show that if $\de\rightarrow0$, then we can find $n(\de)$ such that for every Borel set $A\subset U$, satisfying $l_{\alpha,R}(\partial A)=0$, the above integral tends to $l_{\alpha,R}(A)$.

\emph{Step 2.}
Let $r_\de>0$ be the smallest number such that there exists $n\in \N$ for which
\begin{equation}\label{eq:r}
f_{\de}^{n+1}(\mc N_{\de,R}^+)\subset \overline{B(p_{\de},r_\de)},\:\:\textrm{ and } \:\:
f_{\de}^{n+2}(\mc N_{\de,R}^+)\:\;/\:\!\!\!\!\!\!\!\subset\overline{B(p_{\de},r_{\textrm{z}}/2)}.
\end{equation}
In particular we see that $r_\de\leeq r_{\textrm{z}}/2$, and possible values of $r_\de$ are separated from $0$. It follows from Proposition \ref{prop:Hmetric} that $\mc N_{\de,R}^+\subset U$ for $\de$ small enough.

Let $(\de_k)$ be a sequence of parameters tending to $0$ for which $r_{\de_k}\rightarrow q$ when $k\rightarrow\infty$. Let $n_k$ be such that (\ref{eq:r}) holds for $\de_k$.

We have $f_\de^{n_k+1}(\sqrt{|\de_k|}U)\subset B(p_\de,r_{\textrm{z}})$. If $z\in f_\de(\sqrt{|\de_k|}U)$, then $-z\in B(p_\de,r_{\textrm{z}})$, so using (\ref{eq:pfp}) we obtain
$$f^{n_k}_{\de_k}(z)=f^{n_k}_{\de_k}(-z)=\Phi_{\de_k}^{-1}\big(\lambda_{\de_k}^{n_k}\Phi_{\de_k}(-z)\big).$$
Therefore, if $z\in U$, then we have
\begin{equation}\label{eq:F}
F_{\alpha,q,k}(z):=f^{n_k+1}_{\de_k}\big(\sqrt{|\de_k|}z\big)= \Phi_{\de_k}^{-1}\big(\lambda_{\de_k}^{n_k}\Phi_{\de_k}(-|\de_k|z^2+2-\de_{k})\big).
\end{equation}

\emph{Step 3.}
We will prove that the sequence $F_{\alpha,q,k}$ converges uniformly on $U$.
Write
$$g_\de(z):=\frac{1}{|\de|}\Phi_{\de}\big(-|\de|z^2+2-\de\big)=\frac{1}{|\de|}\Phi_{\de}\big(p_{\de}-|\de|z^2+2-\de-p_{\de}\big),$$
so we have
\begin{equation}\label{eq:FFFF}
F_{\alpha,q,k}(z)=\Phi^{-1}_{\de_k}\big(|\de_k|\lambda_{\de_k}^{n_k}g_{\de_k}(z)\big).
\end{equation}
Note that
$$2-{\de_k}-p_{\de_k}=-\frac23\de_k+O(|\de_k|^2).$$
Therefore
$$g_{\de_k}(z)= \frac{\Phi_{\de_k}\big(p_{\de_k}+|\de_k|(-z^2-\frac23v+O(|\de_k|))\big)}{|\de_k|},$$
where $v=e^{i\alpha}$.
Since $\Phi_{\de_k}'(p_{\de_k})=1$, $\de_k\rightarrow0$ whereas $U$ is bounded, we conclude that
\begin{equation}\label{eq:FFF}
g_{\de_k}(z)\rightrightarrows -z^2-\frac23v.
\end{equation}

Proposition \ref{prop:Hmetric} and the fact that $\Phi_\de\rightrightarrows\Phi_0$, lead to
$$\diam(g_{\de_k}(\mc N^{\divideontimes,+}_{\de_k,R}))\rightarrow K_1,$$
for some constant $K_1$ (which does not depend on $q$ and is close to $R^2$).
On the other hand
$$\diam(F_{\alpha,q,k}(\mc N^{\divideontimes,+}_{\de_k,R}))\rightarrow q.$$
So, (\ref{eq:FFFF}) and the fact $\arg(\la_{n_k})=O(\de_k)$, lead to
$$|\de_k|\lambda_{\de_k}^{n_k}\rightarrow \kappa\in\R.$$
Thus, (\ref{eq:FFFF}) combined with (\ref{eq:FFF}) and the above, gives us
$$F_{\alpha,q,k}(z)\rightrightarrows \Phi^{-1}_{0}\Big(\kappa\Big(-z^2-\frac23v\Big)\Big):=F_{\alpha,q}(z),$$
where $k\rightarrow\infty$.

\emph{Step 4.} Since $(b^\dt_\alpha)^2=-\frac23v$, we have $F_{\alpha,q}(b^\dt_\alpha)=p_0=2$. Thus, we conclude from Proposition \ref{prop:Hmetric}, and the choice of $(\de_k)$, that
$$F_{\alpha,q}\big|_{H^+_{\alpha,R}}:H^+_{\alpha,R}\rightarrow[2-q,2]\subset \mc J_0.$$
Moreover $F_{\alpha,q}\big|_{H^+_{\alpha,R}}$ is a bijection, and then
\begin{equation}\label{eq:FF}
\int_{F_{\alpha,q}(A)} \big|(F^{-1}_{\alpha,q})'\big|d\omega_{0}=l_{\alpha,R}(A).
\end{equation}
Hence, the assumption $l_{\alpha,R}(\partial A)=0$ leads to $\omega_0(F_{\alpha,q}(\partial A))=0$.

We have (cf. (\ref{eq:F}))
$$F_{\alpha,q,k}^{-1}(z)=\frac{f_{\de_k}^{-n_k-1}(z)}{\sqrt{|\de_k|}}\rightarrow F_{\alpha,q}^{-1}(z).$$
Therefore (\ref{eq:omegaint}) gives us
\begin{equation*}
{\omega}^\divideontimes_{\delta_k}(A)=\int_{F_{\alpha,q,k}(A)} \big|(F^{-1}_{\alpha,q,k})'\big|^{d(\de_k)}d\omega_{\de_k}.
\end{equation*}
Since $\omega_0(F_{\alpha,q}(\partial A))=0$
and $F_{\alpha,q,k}$ converges uniformly to $F_{\alpha,q}$, measure of symmetric difference $\omega_{\de_k}(F_{\alpha,q,k}(A) \Delta F_{\alpha,q}(A))$ tends to $0$. So uniform convergence of $|(F^{-1}_{\alpha,q,k})'|^{d(\de_k)}$, Proposition \ref{prop:o} and (\ref{eq:FF}) lead to
$$\int_{F_{\alpha,q,k}(A)} \big|(F^{-1}_{\alpha,q,k})'\big|^{d(\de_k)}d\omega_{\de_k}\\
\rightarrow\int_{F_{\alpha,q}(A)} \big|(F^{-1}_{\alpha,q})'\big|d\omega_{0}=l_{\alpha,R}(A).$$
Thus, the statement holds if $r(\de_k)\rightarrow q$. But, $l_{\alpha,R}$ is the only possible limit, because for every $\de_k$ we can always find a subsequence $\de_{j_k}$ such that $r(\de_{j_k})\rightarrow q$.
\end{proof}

\section{Key integral}\label{sec:int1}

\subsection{}
The main problem in the proof of the Theorem \ref{thm:minus2}, is to estimate integral from the numerator of the formula (\ref{eq:d}), namely:
\begin{equation}\label{eq:integral}
\int_{\mc J_{\de_0}}\frac{\partial}{\partial t}\log|f_{tv}'(\vp_{tv})|d\tilde\mu_{tv}.
\end{equation}

Put $\dot\vp_\delta:=\frac{\partial}{\partial \delta}\vp_\delta$. We have
\begin{equation*}
\frac{\partial}{\partial t}(f'_{tv}(\vp_{tv}))=\frac{\partial}{\partial t}(2\vp_{tv})= 2v\dot\vp_\delta\big|_{\delta=tv}.
\end{equation*}
Therefore, the integrand can be rewritten as follows:
\begin{equation}\label{eq:f1}
\frac{\partial}{\partial t}\log|f'_{tv}(\vp_{tv})|=\re\Big(\frac{\frac{\partial}{\partial t}(f'_{tv}(\vp_{tv}))}{f'_{tv}(\vp_{tv})}\Big)=\re\Big(v\frac{\dot\vp_\de\big|_{\delta=tv}}{\vp_{tv}}\Big).
\end{equation}

\subsection{}
Now we derive the formula for $\dot\vp_\de$.
The function $\vp_\de$ conjugates $f_{\de_0}=T$ to $f_\de$, so $\vp_\de(T(s))=\vp_\de^2(s)-2+\de$.
Differentiating both sides with respect to $\de$, we get
$$\dot{\vp}_\de(T(s))=2\vp_\de(s)\dot{\vp}_\de(s)+1,$$
hence
\begin{equation}\label{eq:s0}
    \dot\vp_\de(s)=-\frac{1}{2\vp_\de(s)}+\frac{\dot\vp_\de(T(s))}{2\vp_\de(s)}.
\end{equation}
Next, replacing $s$ by $T(s)$, $T^2(s)$,..., $T^{m-1}(s)$, we obtain
\begin{align*}\label{eq:sumass}
   \dot\vp_\de(s)=&-\sum_{k=0}^{m-1}
   \frac{1}{2\vp_\de(s)\cdot2\vp_\de(T(s))\cdot...\cdot2\vp_\de(T^{k}(s))}+\\
   &+\frac{\dot\vp_\de(T^m(s))}{2\vp_\de(s)\cdot2\vp_\de(T(s))\cdot...\cdot2\vp_\de(T^{m-1}(s))}.
\end{align*}
Since $2\vp_\de(s)=f'_\de(\vp_\de(s))$, the chain rule leads to
\begin{equation}\label{eq:sumas}
   \dot\vp_\delta(s)=-\sum_{k=1}^{m}\frac{1}{(f_\delta^k)'(\vp_\delta(s))}+\frac{\dot\vp_\delta(T^m(s))}{(f_\delta^m)'(\vp_\delta(s))}.
\end{equation}

\subsection{}
Let us assume that $z\in\mc J_\de$ is close to $0$, and $z=\vp_\de(s)$. Then the denominator of (\ref{eq:f1}) is small, moreover we see from (\ref{eq:s0}) that
   $$\dot\vp_\de(s)=-\frac{1}{2z}(1-\dot\vp_\de(T(s))).$$
Since $\vp_\de(T(s))$ is close to $-2$, it follows from the formula (\ref{eq:sumas}) that $\dot\vp_\de(T(s))$ should usually be close to $1/3$. If this is so, we get (cf. (\ref{eq:f1}))
\begin{equation*}\label{eq:s43}
   \re\Big(v\frac{\dot\vp_\de}{\vp_{\de}}\Big)\approx\frac13\re\Big(-\frac{v}{z^2}\Big).
\end{equation*}
Thus, we can expect that the function $d'_v(tv)$, and the integral of $\re(-v/z^2)$
over a neighborhood of $0$, with respect to $\mu_\de$, have the same order.

Indeed, integral of $\re(-v/z^2)$ has decisive influence on $d'_v(tv)$, and we will compute it over the sets $\mc N_{\de,R}$ in Proposition \ref{prop:intJR}. But first, for $\alpha\in(0,\pi)$ and $R\greq1$, let us define
\begin{equation}\label{eq:I}
   I_{\alpha,R}:=\sqrt6\cos\alpha\Big(\sqrt{\frac{\sin(2\gamma_{\alpha,R})}{\sin\alpha}}-1\Big)+ \frac{\sqrt6}{2}\sqrt{\sin\alpha}\int_{\alpha}^{\pi-2\gamma_{\alpha,R}} \sqrt{\sin \beta}\,d\beta,
\end{equation}
where $\gamma_{\alpha,R}$ is given by the formula (\ref{eq:gamma}). For $\alpha=\pi$ we take
\begin{equation}\label{eq:II}
   I_{\pi,R}={\sqrt6}\Big(1-\frac{\sqrt6}{3}\frac1R\Big)={\sqrt6}-\frac2R.
\end{equation}
Note that
   $$\sqrt{\frac{\sin(2\gamma_{\alpha,R})}{\sin\alpha}}=\frac{\sqrt{6}}{3}\frac1R\frac{1}{\sqrt{1+(\frac{\sin\alpha}{3R^2})^2}},$$
therefore $I_{\alpha,R}\rightarrow I_{\pi,R}$, when $\alpha\rightarrow\pi^-$.

\begin{prop}\label{prop:intJR}
For every $\alpha\in(0,\pi]$ and $R\greq1$ we have
   $$\lim_{\de\rightarrow0}|\de|^{1-\frac12{d(\de)}}\int_{\mc N_{\de,R}}\re\Big(-\frac{v}{z^2}\Big)d\mu_{\de}(z)= \frac2\pi I_{\alpha,R},$$
where $\alpha=\arg\de$ and $v=e^{i\alpha}$.
\end{prop}

\begin{proof}
Fix $\alpha\in(0,\pi)$ and $R\greq1$.
First, using Proposition \ref{prop:omega}, we obtain
\begin{multline*}
\frac{|\de|}{\sqrt{|\de|}^{d(\de)}}\int_{\mc N_{\de,R}}\re\Big(-\frac{v}{z^2}\Big)d\omega_{\de}(z)=\frac{1}{\sqrt{|\de|}^{d(\de)}}\int_{\mc N_{\de,R}}\re\Big(-\frac{v}{(z/\sqrt{|\de|})^2}\Big)d\omega_{\de}(z)
\\=\int_{\mc N^\dt_{\de,R}}\re\Big(-\frac{v}{z^2}\Big)d\omega^\dt_{\de}(z)\rightarrow\int_{ H^\dt_{\alpha,R}}\re\Big(-\frac{v}{z^2}\Big)dl_{\alpha,R}(z).
\end{multline*}
So, taking into account Lemma \ref{lem:mu/omega}, we have to prove that the integral over $H^\dt_{\alpha,R}$ is equal to $I_{\alpha,R}$. It is enough to consider the integral restricted to $H^{\dt,+}_{\alpha,R}$, because value of the integral over $H^{\dt,-}_{\alpha,R}$ is the same.

We have
\begin{equation}\label{eq:H+}
\int_{H^{\dt,+}_{\alpha,R}}\re\Big(-\frac{v}{z^2}\Big)dl_{\alpha,R}(z)= -\int_{H^{\dt,+}_{\alpha,R}}\frac{(x^2-y^2)\cos\alpha+2xy\sin\alpha}{(x^2+y^2)^2}\,dl_{\alpha,R}(z),
\end{equation}
where $x=\re(z)$ and $y=\im(z)$.

If $\alpha\in(0,\pi)$, then $H^{\dt,+}_{\alpha,R}$ is the arc of hyperbola $xy=-\frac13\sin\alpha$ contained in $IV$ quadrant, which joins the points $b^\dt_\alpha$ and $z^\dt_{\alpha,R}$ (see (\ref{eq:endp})).
So, in the polar coordinates $H^{\dt,+}_{\alpha,R}$ can be written as follows:
$$
H^{\dt,+}_{\alpha,R}:\left\{ \begin{array}{ll}
x(t)=h_\alpha(t)\cdot\cos t,\\
y(t)=h_\alpha(t)\cdot \sin{t}.
\end{array} \right.\, t\in\Big[\arctan\frac{\im b^\dt_\alpha}{\re b^\dt_\alpha},\arctan\frac{\im z^\dt_{\alpha,R}}{\re z^\dt_{\alpha,R}}\Big],
$$
where $h^2_\alpha(t)\cdot\frac12\sin(2t)=-\frac13\sin\alpha$, and
$$\frac{\im b^\dt_\alpha}{\re b^\dt_\alpha}=-\cot\frac\alpha2=\tan\Big(\frac\alpha2-\frac\pi2\Big),\quad\frac{\im z^\dt_{\alpha,R}}{\re z^\dt_{\alpha,R}}=-\frac{1}{3R^2}\sin\alpha.$$
Therefore, we obtain (cf. (\ref{eq:gamma}))
$$
H^{\dt,+}_{\alpha,R}:\left\{ \begin{array}{ll}
x(t)=\sqrt{-\frac23\frac{\sin\alpha}{\sin(2t)}}\cdot\cos t,\\
y(t)=\sqrt{-\frac23\frac{\sin\alpha}{\sin(2t)}}\cdot \sin{t},
\end{array} \right.\, t\in\Big[\frac12(\alpha-\pi),-\gamma_{\alpha,R}\Big].
$$

Next, we can get
$$dl_{\alpha,R}=\sqrt{h_\alpha^2(t)+(h_\alpha'(t))^2}\,dt=\sqrt{-\frac23\frac{\sin\alpha}{\sin^3(2t)}}\,dt.$$
Since $x(t)y(t)=-\frac13\sin\alpha$, and $x^2(t)+y^2(t)=h^2_\alpha(t)$, moreover
$$x^2(t)-y^2(t)=h_\alpha^2(t)\big(\cos^2t-\sin^2t\big)=-\frac23\sin\alpha\frac{\cos(2t)}{\sin(2t)},$$
the integrals from (\ref{eq:H+}) are equal to
$$-\int_{\frac12(\alpha-\pi)}^{-\gamma_{\alpha,R}}\frac{-\frac23\sin\alpha\cos\alpha\frac{\cos(2t)}{\sin(2t)} -\frac23\sin^2\alpha}{\frac49\sin^2\alpha\frac{1}{\sin^2(2t)}} \sqrt{-\frac23\frac{\sin\alpha}{\sin^3(2t)}}\,dt.$$
Simplifying we obtain
$$\frac{\sqrt{6}}{2}\int_{\frac12(\alpha-\pi)}^{-\gamma_{\alpha,R}}\Big(\frac{\cos\alpha}{\sin\alpha}\frac{\cos(2t)}{\sin(2t)}+1\Big) \sin^2(2t)\sqrt{-\frac{\sin\alpha}{\sin^3(2t)}}\,dt.$$

Substitution $\beta=2t+\pi$, and the above computations, lead to
$$\int_{H^{\dt,+}_{\alpha,R}}\re\Big(-\frac{v}{z^2}\Big)dl_{\alpha,R}(z)= \frac{\sqrt{6}}{4}\int_{\alpha}^{\pi-2\gamma_{\alpha,R}}\Big(\frac{\cos\alpha}{\sin\alpha}\frac{\cos \beta}{\sin \beta}+1\Big) \sin^2\beta\sqrt{\frac{\sin\alpha}{\sin^3\beta}}\,d\beta.$$
We divide the integrand into two parts. First we compute:
\begin{multline*}
\frac{\sqrt{6}}{4}\int_{\alpha}^{\pi-2\gamma_{\alpha,R}}\frac{\cos\alpha}{\sin\alpha}\frac{\cos \beta}{\sin \beta} \sin^2\beta\sqrt{\frac{\sin\alpha}{\sin^3\beta}}\,d\beta=
\frac{\sqrt{6}}{4}\frac{\cos\alpha}{\sqrt{\sin\alpha}}\int_{\alpha}^{\pi-2\gamma_{\alpha,R}}\frac{\cos \beta}{\sqrt{\sin \beta}}\,d\beta\\
=\frac{\sqrt{6}}{4}\frac{\cos\alpha}{\sqrt{\sin\alpha}}\big(2\sqrt{\sin\beta}\big)\Big|_\alpha^{\pi-2\gamma_{\alpha,R}} =\frac{\sqrt{6}}{2}\cos\alpha\sqrt{\frac{\sin2\gamma_{\alpha,R}}{\sin\alpha}}-\frac{\sqrt{6}}{2}\cos\alpha.
\end{multline*}
Next we have
$$\frac{\sqrt{6}}{4}\int_{\alpha}^{\pi-2\gamma_{\alpha,R}} \sin^2\beta\sqrt{\frac{\sin\alpha}{\sin^3\beta}}\,d\beta= \frac{\sqrt{6}}{4}\sqrt{\sin\alpha}\int_{\alpha}^{\pi-2\gamma_{\alpha,R}} \sqrt{\sin \beta}\,d\beta.$$
Thus, the integral over $H^{\dt,+}_{\alpha,R}$ is equal $\frac12I_{\alpha,R}$ (cf. (\ref{eq:I})), so the statement holds for $\alpha\in(0,\pi)$.

If $\alpha=\pi$, then $H^{\dt,+}_{\pi,R}:\frac{\sqrt6}{3}\leeq x\leeq R$, and

$$\int_{H^{\dt,+}_{\pi,R}}\re\Big(-\frac{-1}{z^2}\Big)dl_{\pi,R}(z)= \int^{R}_{\frac{\sqrt6}{3}}\frac{1}{x^2}\,dx=-\frac{1}{x}\Big|^{R}_{\frac{\sqrt6}{3}}=\frac{\sqrt{6}}{2}-\frac1R=\frac12I_{\pi,R},$$
(cf. (\ref{eq:II})) thus the proof is finished.
\end{proof}

Now we give two technical Corollaries.

\begin{cor}\label{cor:IntJR2}
For every $\alpha\in(0,\pi]$ and $R\greq1$ there exists $K_{\alpha,R}\in\R$ such that
   $$\lim_{\de\rightarrow0}|\de|^{1-\frac12d(\de)}\int_{\mc N_{\de,R}}\frac{1}{|z|^2}\,d\mu_{\de}(z)= K_{\alpha,R},$$
where $\alpha=\arg\de$.
\end{cor}

\begin{proof}
Analogously as in the proof of Proposition \ref{prop:intJR}, it is enough to compute integral $\int_{ H^{\dt,+}_{\alpha,R}}|z|^{-2}dl_{\alpha,R}(z)$. For $\alpha\in(0,\pi)$, proceeding as before, we obtain
\begin{multline*}
\int_{ H^{\dt,+}_{\alpha,R}}\frac{1}{|z|^2}\,dl_{\alpha,R}(z) =\int_{\frac{1}{2}(\alpha-\pi)}^{-\gamma_{\alpha,R}} -\frac32\frac{\sin(2t)}{\sin\alpha}\sqrt{-\frac23\frac{\sin\alpha}{\sin^3(2t)}}\,dt\\= \sqrt{\frac32}\frac{1}{\sqrt{\sin\alpha}}\int_{\frac{1}{2}(\alpha-\pi)}^{-\gamma_{\alpha,R}}\frac{1}{\sqrt{-\sin(2t)}}\,dt:=\frac12K_{\alpha,R}.
\end{multline*}

If $\alpha=\pi$, then
$$\int^{R}_{\frac{\sqrt{6}}{3}}\frac{1}{x^2}\,dx=\frac{\sqrt{6}}{2}-\frac1R:=\frac12K_{\pi,R},$$
and the proof is finished.
\end{proof}

\begin{cor}\label{cor:IntJR}
For every $\alpha\in(0,\pi]$ and $R\greq1$ there exists $K_{\alpha,R}\in\R$ such that
   $$\lim_{\de\rightarrow0}|\de|^{\frac12(1-d(\de))}\int_{\mc N_{\de,R}}\frac{1}{|z|}\,d\mu_{\de}(z)= K_{\alpha,R},$$
where $\alpha=\arg\de$.
\end{cor}

\begin{proof}
Using Proposition \ref{prop:omega}, analogously as in the proof of Proposition \ref{prop:intJR} we get
\begin{equation*}
\frac{\sqrt{|\de|}}{\sqrt{|\de|}^{d(\de)}}\int_{\mc N_{\de,R}}\frac{1}{|z|}\, d\omega_{\de}(z)
=\int_{\mc N^\dt_{\de,R}}\frac{1}{|z|}\,d\omega^\dt_{\de}(z)\rightarrow\int_{ H^\dt_{\alpha,R}}\frac{1}{|z|}\,dl_{\alpha,R}(z).
\end{equation*}

For $\alpha\in(0,\pi)$, proceeding as before, we obtain
\begin{multline*}
\int_{ H^{\dt,+}_{\alpha,R}}\frac{1}{|z|}\,dl_{\alpha,R}(z) =\int_{\frac12(\alpha-\pi)}^{-\gamma_{\alpha,R}} \sqrt{-\frac32\frac{\sin(2t)}{\sin\alpha}}\sqrt{-\frac23\frac{\sin\alpha}{\sin^3(2t)}}\,dt\\= \int_{\frac12(\alpha-\pi)}^{-\gamma_{\alpha,R}}\frac{1}{|\sin(2t)|}\,dt=-\frac12\log|\tan t|\bigg|_{\frac12(\alpha-\pi)}^{-\gamma_{\alpha,R}}\\=\frac12\log\frac32+\log\frac{R}{\sin\frac\alpha2}=:\frac12K_{\alpha,R}.
\end{multline*}

If $\alpha=\pi$, we get
$$\int^{R}_{\frac{\sqrt{6}}{3}}\frac{1}{x}\,dx=\log R+\frac12\log\frac32=:\frac12K_{\pi,R},$$
and the proof is finished.
\end{proof}

\section{Integral over $\mc J_{\de_0}\sms N_R$}\label{sec;int2}

The main result of this Section is Proposition \ref{prop:int}, which shows us that the integral of (\ref{eq:f1}) over $\mc J_{\de_0}\sms N_R$ is small (after dividing by $|\de|^{1-\frac12d(\de)}$).
This result, together with Proposition \ref{prop:intphiJR} (integral over $N_R$) are the main ingredients in the proof of Theorem \ref{thm:minus2}.

But, we begin with a few facts about $(f_\de^n)'$, and Proposition \ref{prop:|phi|}, which will be used in the proofs of both mentioned Propositions.

Because $f_\de$ are conjugated to $z\mapsto \la_\de z$ on the set $B(2,r_\textrm{z})$ we conclude that:
\begin{lem}\label{lem:4^n}
For every $\alpha\in(0,2\pi)$ there exists $K>1$ such that if $z\in\mc C_{\de,n}^{\pm2}$, $n\greq1$ and $1\leeq j\leeq n$, then
$$K^{-1}|\la_\de|^j<|(f_\de^j)'(z)|<K|\la_\de|^j,$$
where $\alpha=\arg\de$ and $0<|\de|<\eta$.
\end{lem}

If $U$ is a neighborhood of $0$ and $z\in\mc J_\de\sms U$, then $f_\de'(z)$ is close to $f_0'(\re z)$, where $\re z\in\mc J_0$ provided $|\re z|\leeq2$. The cylinders $C_{\de,\nu}\ni z$ and $C_{0,\nu}$ are close each other (cf. Lemma \ref{lem:Cnu}), so also finite trajectories $f_\de^n(z)$ and $f_0^n(\re z)$. Thus the formula (\ref{eq:(fn)'}) leads to:

\begin{lem}\label{lem:4-x^2}
For every $\alpha\in(0,2\pi)$, $\ve>0$, $N\in\N$ and open set $U\ni0$, there exists $\eta>0$ such that if $z\in\mc J_\de$, $|\re(z)|\leeq2$, $\{z,f_\de(z),\ldots,f_\de^{n-1}(z)\}\cap U=\emptyset$, $n\leeq N$, then
$$(1-\ve)\,2^n\sqrt{\frac{4-(\re f_0^n(z))^2}{4-(\re z)^2}}<|(f_\de^n)'(z)|<(1+\ve)\,2^n\sqrt{\frac{4-(\re f_0^n(z))^2}{4-(\re z)^2}},$$
where $0<|\de|<\eta$ and $\alpha=\arg\de$.
\end{lem}

Lemma \ref{lem:4-x^2} for $n=i$, Lemma \ref{lem:mm} (\ref{lit:mm1}), and Lemma \ref{lem:4^n} for $n=j-1$ lead to:

\begin{cor}\label{cor:K}
For every $\alpha\in(0,2\pi)$, $N\in\N$ there exists $K>1$ and $\eta>0$ such that if $z\in f_\de^{-i}(\mc C^{0}_{\de,m})$, $m\greq0$, then
$$|(f_\de^{i+j})'(z)|>K^{-1}2^i|\la_\de|^{j-1-\frac{m}{2}},$$
where $0\leeq i \leeq N$, $0\leeq j\leeq m$, $0<|\de|<\eta$ and $\alpha=\arg\de$.
\end{cor}

\begin{prop}\label{prop:|phi|}
For every $\alpha\in(0,\pi]$ and $s>0$, we have
   $$\lim_{\de\rightarrow0}|\de|^s\int_{\mc J_{\de_0}}\big|\dot\vp_\de\big|\,d\tilde{\mu_{\de}}= 0,$$
where $\alpha=\arg\de$.
\end{prop}

\begin{proof}
Fix $\alpha\in(0,\pi]$, $s>0$ and $N_0\greq4$ such that $2^{N_0}\greq8K$, where $K$ is a constant from Corollary \ref{cor:K}.

\emph{Step 1.} First, we define the set $\mc X_{\de,N_0}$, on which it can be easily proven that $|(f_\de^{N_0})'|>10$.

If $f_\de^{N_0}(z)\in \mc M^0_{\de,1}$ (note that $\mc M^0_{0,1}=[-\sqrt2,\sqrt2]$), then Lemma \ref{lem:4-x^2} gives us
\begin{equation*}
|(f_\de^{N_0})'(z)|>2^{N_0}\frac{9}{10}\,\sqrt{\frac24}\greq2^4\frac{9}{10}\,\frac{\sqrt2}{2}>10.
\end{equation*}
The derivative will be also grater than $10$ for $z\in \mc C_{\de,N_0+n}^{\pm2}$, $n\greq1$ and $z=\pm p_\de$, therefore we define
$$\mc X_{\de,N_0}:=f_\de^{-N_0}(\mc M^0_{\de,1})\cup\bigcup_{n=1}^{\infty}\mc C_{\de,N_0+n}^{\pm2}\cup\{-p_\de,p_\de\}.$$
Write $X_{N_0}:=\vp_\de^{-1}(\mc X_{\de,N_0})$.
Let $s\in X_{\de,N_0}$, then (\ref{eq:sumas}) for $m=N_0$, gives us:
$$\dot\vp_\de(s)=-\sum_{k=1}^{N_0}\frac{1}{(f_\de^k)'(\vp_\de(s))}+\frac{\dot\vp_\de(T^{N_0}(s))}{(f_\de^{N_0})'(z)}.$$
Since $f_\de^{N_0}(z)\in\mc M^0_{\de,1}$, the trajectory $\{z,f_\de(z),\ldots,f_\de^{{N_0}-1}(z)\}$ is disjoint from $\mc M^0_{\de,N_0}$, hence is separated from $0$.
Thus, the derivatives $(f_\de^k)'(\vp_\de(s))$ are also separated from $0$ (by a constant depending on $N_0$). So, the fact that the measure $\tilde{\mu_\de}$ is $T$-invariant, leads to
\begin{multline}\label{eq:s1}
\int_{X_{N_0}}|\dot\vp_\de|\,d\tilde{\mu_\de}<K_1(N_0)+\frac{1}{10}\int_{X_{N_0}}|\dot\vp_\de(T^{N_0})|\,d\tilde{\mu_\de}\\
<K_1(N_0)+\frac{1}{10}\int_{T^{N_0}(X_{N_0})}|\dot\vp_\de|\,d\tilde{\mu_\de} <K_1(N_0)+\frac{1}{10}\int_{\mc J_{\de_0}}|\dot\vp_\de|\,d\tilde{\mu_\de}.
\end{multline}

\emph{Step 2.} Now, we will deal with the set
   $$\mc J_{\de}\sms \mc X_{\de,N_0}=\bigcup_{n=1}^\infty (f_\de^{-N_0}(\mc C_{\de,n}^{\pm2})\sms\mc C_{\de,n+N_0}^{\pm2})=\bigcup_{n=1}^\infty \mc G_{\de,n}^{-N_0},$$
where $\mc G_{\de,n}^{-N_0}:=f_\de^{-N_0}(\mc C_{\de,n}^{\pm2})\sms\mc C_{\de,n+N_0}^{\pm2}$. Let $G_{\de,n}^{-N_0}:=\vp_{\de}^{-1}(\mc G_{\de,n}^{-N_0})$.

If $\vp_\de(s)=z\in\mc G_{\de,n}^{-N_0}$, then we will use the formula (\ref{eq:sumas}) for $m=N_0+n$:
\begin{equation}\label{eq:ff}
\dot\vp_\de(s)=-\sum_{j=1}^{N_0+n}\frac{1}{(f_\de^j)'(\vp_\de(s))}+\frac{\dot\vp_\de(T^{N_0+n}(s))}{(f_\de^{N_0+n})'(z)}.
\end{equation}

Before we pass to estimations, let us rewrite the set $\mc G_{\de,n}^{-N_0}$.
We have
\begin{equation*}
f_\de^{-{N_0}}(\mc C_{\de,n}^{-2})=f_\de^{-(N_0-1)}(f_\de^{-1}(\mc C_{\de,n}^{-2}))=f_\de^{-(N_0-1)}(\mc C^{0}_{\de,n+1}),
\end{equation*}
and next
\begin{multline*}
f_\de^{-{N_0}}(\mc C_{\de,n}^{+2})=\mc C_{\de,n+N_0}^{+2}\cup\bigcup_{k=0}^{N_0-1} f_\de^{-({N_0}-k-1)}(f_{\de,-}^{-1}(\mc C_{\de,n+k}^{+2}))\\
=\mc C_{\de,n+N_0}^{+2}\cup  \mc C_{\de,n+N_0}^{-2}\cup\bigcup_{k=0}^{N_0-2} f_\de^{-({N_0}-k-2)}(\mc C^{0}_{\de,n+k+2}).
\end{multline*}
Therefore
\begin{equation}\label{eq:Gn}
\mc G_{\de,n}^{-N_0}=\bigcup_{k=-1}^{N_0-2} f_\de^{-({N_0}-k-2)}(\mc C^{0}_{\de,n+k+2})=\bigcup_{k=1}^{N_0} f_\de^{-({N_0}-k)}(\mc C^{0}_{\de,n+k}).
\end{equation}

\emph{Step 3.} Now we will estimate integral of the "tail" of (\ref{eq:ff}).
If $z\in\mc G_{\de,n}^{-N_0}$, then we see that there exists $k\in\{1,\ldots,N_0\}$ such that $f_\de^{N_0-k}(z)\in\mc C_{\de,n+k}$. Thus,
Corollary \ref{cor:K} for $N=N_0-1$, $i=N_0-k$, $m=n+k$ and $j=n+k$, leads to
\begin{align*}
|(f_\de^{{N_0}+n})'(z)|&>K^{-1}2^{N_0-k}|\la_\de|^{\frac12(n+k)-1}\\
&=K^{-1} 2^{N_0}|\la_\de|^{\frac{n}{2}-1}\Big(\frac{\sqrt{|\la_\de|}}{2}\Big)^k
>\frac{9}{40}K^{-1}2^{N_0} \Big(\frac{19}{10}\Big)^{n}.
\end{align*}
Therefore
\begin{align*}
    \int_{G_n^{-N_0}}\bigg|\frac{\dot\vp_\de(T^{N_0+n})}{(f_\de^{N_0+n})'(\vp_\de)}\bigg|\,d\tilde{\mu_\de} &<\frac{40}{9}K2^{-N_0}\Big(\frac{10}{19}\Big)^{n}\int_{ G_{n}^{-N_0}}|\dot\vp_\delta(T^{{N_0}+n})|\,d\tilde{\mu_\de}\\
&<\frac{40}{9}K2^{-N_0}\Big(\frac{10}{19}\Big)^{n}\int_{\mc J_{\de_0}}|\dot\vp_\delta|\,d\tilde{\mu_\de}.
\end{align*}
By assumption, $2^{N_0}>8K$, we get
$$\frac{40}{9}K2^{-N_0}\sum_{n=1}^{\infty}\Big(\frac{10}{19}\Big)^{n}= 4K2^{-N_0}\Big(\frac{10}{9}\Big)^2<\frac12\Big(\frac{10}{9}\Big)^2<\frac58.$$
Thus we obtain
\begin{equation}\label{eq:s3}
\sum_{n=1}^{\infty}\int_{G_n^{-N_0}}\bigg|\frac{\dot\vp_\de(T^{N_0+n})}{(f_\de^{N_0+n})'(\vp_\de)}\bigg|\,d\tilde{\mu_\de} <\frac58\int_{\mc J_{\de_0}}|\dot\vp_\delta|\,d\tilde{\mu_\de}.
\end{equation}

\emph{Step 4.}
We have to estimate finite sum from (\ref{eq:ff}). If $\vp_\de(s)=z\in\mc G_{\de,n}^{-N_0}$ and $f_\de^{N_0-k}(z)\in\mc C^{0}_{\de,n+k}$ (cf. (\ref{eq:Gn})), then we can write
\begin{equation}\label{eq:fff}
-\sum_{j=1}^{N_0+n}\frac{1}{(f_\de^j)'(z)}=-\sum_{j=1}^{{N_0}-k}
\frac{1}{(f_\delta^j)'(z)}-\sum_{j=N_0-k+1}^{N_0+n}
\frac{1}{(f_\de^{j})'(z)}.
\end{equation}

Now we will estimate first sum on the right. If $k=N_0$, then we will assume that it is equal to $0$.

Since $f_\de^{{N_0}-k}(z)\in \mc C^{0}_{\de,n+k}$, the trajectory $\{z,f_\de(z),\ldots,f_\de^{{N_0}-k-1}(z)\}$ is disjoint from $\mc M^0_{\de,N_0}$, hence is separated from $0$. Therefore the derivatives are separated from $0$, and the sum is bounded by a constant $K_2(N_0)/4$. Thus we have
\begin{equation}\label{eq:s4}
\sum_{n=1}^{\infty}\sum_{k=1}^{N_0}\int_{f_\de^{-{N_0}+k}(\mc C^{0}_{\de,n+k})}\sum_{j=1}^{{N_0}-k}
\bigg|\frac{1}{(f_\delta^j)'(z)}\bigg|\,d\mu_\de(z)<K_2(N_0).
\end{equation}

\emph{Step 5.} The rightmost sum from (\ref{eq:fff}) can be rewritten in the form
\begin{multline*}
-\sum_{j=N_0-k+1}^{N_0+n}
\frac{1}{(f_\de^{j})'(z)}=-\sum_{j=1}^{n+k}
\frac{1}{(f_\de^{j})'(f_\de^{N_0-k}(z))\cdot(f_\de^{N_0-k})'(z)}\\=
-\frac{1}{(f_\de^{N_0-k})'(z)}\cdot\frac{1}{f_\de'(f_\de^{N_0-k}(z))}\bigg(1+\sum_{j=1}^{n+k-1}
\frac{1}{(f_\de^{j})'(f_\de^{N_0-k+1}(z))}\bigg).
\end{multline*}

If $f_\de^{N_0-k}(z)\in\mc C^{0}_{\de,n+k}$ then $|(f_\de^{N_0-k})'(z)|>1$ (cf. Lemma \ref{lem:4-x^2}).
Moreover
$f_\de^{N_0-k+1}(z)\in\mc C_{\de,n+k-1}^{-2}$, so using Lemma \ref{lem:4^n}, we can find $K_3>1$ such that
\begin{equation}\label{eq:n-k+1}
   \bigg|\sum_{j=N_0-k+1}^{N_0+n}
   \frac{1}{(f_\de^{j})'(z)}\bigg|<2K_3\frac{1}{|f_\de'(f_\de^{N_0-k}(z))|}=K_3\frac{1}{|f_\de^{N_0-k}(z)|}.
\end{equation}

\emph{Step 6.} In order to estimate the above expression, we will consider two cases.
First, we will deal with the sets $\mc G_{\de,n}^{-N_0}$ such that if $z\in\mc G_{\de,n}^{-N_0}$, then the trajectory $\{z,\ldots,f_\de^{N_0-1}(z)\}$ is disjoint from $\mc N_{\de,1}$. It is enough to check if $z\in\mc C^{0}_{\de,n+N_0}$ then $|\re(z)|>\sqrt{|\de|}$. Since $|z|< K|\lambda_\de|^{-\frac12(n+N_0)}$ (see Lemma \ref{lem:mm} (\ref{lit:mm2})), we get $\sqrt{|\de|}< K|\lambda_\de|^{-\frac12(n+N_0)}$, and then $1\leeq n<\log_{|\lambda_\de|}\big(\frac{K^2}{|\de|}\big)-N_0$.

We have
\begin{equation}\label{eq:f-N_0}
   K_3\int_{f_\de^{-{N_0}+k}(\mc C^{0}_{\de,n+k})} \frac{1}{|f_\de^{N_0-k}(z)|}\,d \mu_\de(z)=K_3\int_{\mc C^{0}_{\de,n+k}}\frac{1}{|z|}\,d \mu_\de(z).
\end{equation}
So, using (\ref{eq:n-k+1}), (\ref{eq:f-N_0}), the fact that $\mu_\de(\mc C^{0}_{\de,n})\asymp\omega_\de(\mc C^{0}_{\de,n})\asymp(\diam(\mc C^{0}_{\de,n}))^{d(\de)}$ and Lemma \ref{lem:mm} we obtain
\begin{multline*}
   \sum_{k=1}^{N_0}\int_{f_\de^{-{N_0}+k}(\mc C^{0}_{\de,n+k})}\bigg|\sum_{j=N_0-k+1}^{N_0+n}
   \frac{1}{(f_\de^{j})'(z)}\bigg|d{\mu_\de}(z)<
   K_3\sum_{k=1}^{N_0}\int_{\mc C^{0}_{\de,n+k}}\frac{1}{|z|}d \mu_\de(z)\\<K_4\sum_{k=1}^{N_0}|\la_\de|^{\frac12(n+k)}|\la_\de|^{-\frac12d(\de)(n+k)} < K_5({N_0})|\la_\de|^{\frac{1}{2}(1-d(\de))n}.
\end{multline*}
We can assume that $1-d(\de)\leeq|1-d(\de)|<s/2$, thus the integral over all sets $\mc G_{\de,n}^{-N_0}$, where $1\leeq n\leeq\log_{|\lambda_\de|}\big(\frac{K^2}{|\de|}\big)-N_0$ can be estimated by
\begin{multline}\label{eq:s5}
   K_5(N_0)\sum_{n=1}^{\log_{|\lambda_\de|}\big(\frac{K^2}{|\de|}\big)-N_0}|\lambda_\de|^{\frac12|1-d(\de)|n}< K_5(N_0)\Big(\frac{K^2}{|\de|}\Big)^{\frac{s}{4}}\log_{|\lambda_\de|}\Big(\frac{K^2}{|\de|}\Big) \\<|\de|^{-\frac{s}{4}}|\de|^{-\frac{s}{4}}
=|\de|^{-\frac{s}{2}}.
\end{multline}

\emph{Step 7.} Let $n_\de$ be the smallest number such that there exists $z\in\mc G_{\de,n_\de}^{-N_0}$ for which $\{z,\ldots,f_\de^{N_0-1}(z)\}\cap\mc N_{\de,1}\neq\emptyset$, hence $\mc C^{0}_{\de,n_\de+N_0}\cap\mc N_{\de,1}\neq\emptyset$. Thus, there exists $R>1$ ($R\asymp 2^{N_0}$) such that $\mc M^0_{\de,n_\de}\subset\mc N_{\de,R}$. So, using (\ref{eq:Gn}), we obtain
   $$\bigcup_{n=n_\de}^{\infty}\mc G_{\de,n}^{-N_0}\subset\bigcup_{k=1}^{N_0}f_\de^{-N_0+k}(\mc M^0_{\de,n_\de})\subset\bigcup_{k=1}^{N_0}f_\de^{-N_0+k}(\mc N_{\de,R}).$$

Using (\ref{eq:n-k+1}), (\ref{eq:f-N_0}) (analogously as in the \emph{Step 5}), and next Corollary \ref{cor:IntJR}, we obtain
\begin{multline}\label{eq:s6}
\sum_{k=1}^{N_0}\int_{f_\de^{-N_0+k}(\mc N_{\de,R})}\bigg|\sum_{j=N_0-k+1}^{N_0+n}
\frac{1}{(f_\de^{j})'(z)}\bigg|d{\mu_\de}(z)
\\
<K_3\sum_{k=1}^{N_0}\int_{f_\de^{-N_0+k}(\mc N_{\de,R})} \frac{1}{|f_\de^{N_0-k}(z)|}\,d \mu_\de(z)=K_3\cdot N_0\int_{\mc N_{\de,R}}\frac{1}{|z|}\,d \mu_\de(z)<|\de|^{-\frac{s}{2}}.
\end{multline}

\emph{Step 8.}
Estimates (\ref{eq:s1}), (\ref{eq:s3}), (\ref{eq:s4}), (\ref{eq:s5}) and (\ref{eq:s6}) give us
\begin{equation*}
    \int_{\mc J_{\de_0}}\big|\dot\vp_\de\big|\,d\tilde{\mu_\de}<K_1(N_0)+K_2(N_0)+\Big(\frac{1}{10}+\frac58\Big) \int_{\mc J_{\de_0}}\big|\dot\vp_\de\big|\,d\tilde{\mu_\de} +2|\de|^{-\frac{s}{2}}.
\end{equation*}
Since $\frac{1}{10}+\frac{5}{8}<\frac34$, we obtain
\begin{equation*}
    \frac14\int_{\mc J_{\de_0}}\big|\dot\vp_\de\big|\,d\tilde{\mu_\de}<K_1(N_0)+K_2(N_0) +2|\de|^{-\frac{s}{2}},
\end{equation*}
and the statement follows.
\end{proof}

\begin{prop}\label{prop:int}
For every $\alpha\in(0,\pi]$, $\ve>0$ there exist $R\greq1$ and $\eta>0$ such that
   $$|\delta|^{1-\frac12d(\de)}\int_{\mc J_{\de_0}\sms N_{R}}\Big|\frac{\dot\vp_\de}{\vp_\de}\Big|\,d\tilde{\mu_{\de}}<\ve,$$
where $0<|\de|<\eta$ and $\alpha=\arg\de$.
\end{prop}

\begin{proof}
We conclude from Proposition \ref{prop:|phi|} that it is enough to estimate the integral over the set $M^0_1\sms N_{R}$.

Fix $\alpha\in(0,\pi]$ and $\ve>0$.
If $\vp_\de(s)=z\in\mc C^{0}_{\de,n}$, then formula (\ref{eq:sumas}) for $m=n$ leads to
%\begin{equation*}\label{eq:sc}
%\dot\vp_\delta(s)=-\frac{1}{f_\delta'(\vp_\delta(s))}\bigg(1+\sum_{k=1}^{n-1}
%\frac{1}{(f_\delta^{k})'(f_\de(\vp_\delta(s)))}\bigg)+\frac{\dot\vp_\delta(T^{n}(s))}{(f_\delta^{n})'(\vp_\delta(s))}.
%\end{equation*}
\begin{equation*}\label{eq:sc1}
\frac{\dot\vp_\delta(s)}{\vp_\de(s)}=-\frac{1}{z\cdot f_\delta'(z)}\bigg(1+\sum_{k=1}^{n-1}
\frac{1}{(f_\delta^{k})'(f_\de(z))}\bigg)+\frac{\dot\vp_\delta(T^{n}(s))}{\vp_\delta(s)\cdot(f_\delta^{n})'(\vp_\delta(s))}.
\end{equation*}
where $\alpha=\arg\de$.

We have $|z|>R\sqrt{|\de|}$ for some $R$ large enough, on the other hand Lemma \ref{lem:mm} gives us $|z|<K_1|\la_\de|^{-\frac{1}{2}n}$. So we will consider $1\leeq n\leeq\log_{|\lambda_\de|}\big(\frac{K_1^2}{R^2|\de|}\big)$.

Since $f_\de(z)\in\mc C_{\de,n-1}^{-2}$, using Lemma \ref{lem:4^n} and again Lemma \ref{lem:mm}, we get
\begin{equation*}
   |z\cdot (f_\delta^{n})'(z)|=|z\cdot f_\de'(z)\cdot(f_\delta^{n-1})'(f_\de(z))|>K_2|\la_\de|^{-n}|\la_\de|^{n-1}>K_2.
\end{equation*}
So, the above inequality, and the fact that $T^n(C^{0}_n)=M^0_1$ lead to
\begin{equation*}
   \sum_{n=1}^{\log_{|\lambda_\de|}\big(\frac{K_1^2}{R^2|\de|}\big)}\int_{C^{0}_n} \bigg|\frac{\dot\vp_\delta(T^{n})}{\vp_\de\cdot(f_\delta^{n})'(\vp_\delta)}\bigg|\,d\tilde{\mu_\de} < K_3\log_{|\lambda_\de|}\Big(\frac{K_1^2}{R^2|\de|}\Big)\int_{M^0_1} \big|{\dot\vp_\delta}\big|\,d\tilde{\mu_\de}.
\end{equation*}
Because $1-d(\de)/2$ is close to $1/2$, Proposition \ref{prop:|phi|} gives us
\begin{equation*}
   \lim_{\de\rightarrow0}|\delta|^{1-\frac12d(\de)}K_3\log_{|\lambda_\de|}\Big(\frac{K_1^2}{R^2|\de|}\Big)\int_{M^0_1} \big|{\dot\vp_\delta}\big|\,d\tilde{\mu_\de}=0.
\end{equation*}

Since $f_\de(z)\in\mc C_{\de,n-1}^{-2}$, using Lemma \ref{lem:4^n}, next Lemma \ref{lem:mm}, and the fact that $\mu_\de(\mc C^{0}_{\de,n})\asymp\omega_\de(\mc C^{0}_{\de,n})\asymp(\diam(\mc C^{0}_{\de,n}))^{d(\de)}$, we obtain
\begin{multline*}
  \sum_{n=1}^{\log_{|\lambda_\de|}\big(\frac{K_1^2}{R^2|\de|}\big)}\int_{\mc C^{0}_{\de,n}} \bigg|\frac{1}{z\cdot f_\delta'(z)}\bigg(1+\sum_{k=1}^{n-1}
\frac{1}{(f_\delta^{k})'(f_\de(z))}\bigg)\bigg|\,d{\mu_\de}\\
<K_4\sum_{n=1}^{\log_{|\lambda_\de|}\big(\frac{K_1^2}{R^2|\de|}\big)}\int_{\mc C^{0}_{\de,n}} \frac{1}{|z|^2}\,d{\mu_\de} <K_5\sum_{n=1}^{\log_{|\lambda_\de|}\big(\frac{K_1^2}{R^2|\de|}\big)}|\lambda_\de|^{(1-\frac12d(\de))n}.
\end{multline*}
The rightmost sum can be estimated by the integral:
$$ K_6\int_{1}^{\log_{|\lambda_\de|}\big(\frac{K_1^2}{R^2|\de|}\big)}|\lambda_\de|^{(1-\frac12d(\de))x}dx< K_7\Big(\frac{K_1^2}{R^2|\de|}\Big)^{1-\frac12d(\de)}< \ve |\delta|^{-1+\frac12d(\de)}.$$
The last inequality holds for $R$ large enough, since $K_1$ and $K_7$ does not depend on $R$. Thus the assertion follows.
\end{proof}

\section{Integral over $N_R$ and Proof of the main Theorem}\label{sec:int3}

Now, using Proposition \ref{prop:intJR}, we estimate integral of (\ref{eq:f1}) over $N_{R}$, which has decisive influence on (\ref{eq:integral}). Next we will prove Theorem \ref{thm:minus2}.

\begin{prop}\label{prop:intphiJR}
For every $\alpha\in(0,\pi]$ and $R\greq1$ we have
   $$\lim_{\de\rightarrow0}|\de|^{1-\frac12d(\de)}\int_{N_{R}}\re\Big(v\frac{\dot\vp_\de}{\varphi_\de}\Big)d\tilde{\mu_{\de}}= \frac13\frac2\pi I_{\alpha,R},$$
where $\alpha=\arg\de$ and $v=e^{i\alpha}$.
\end{prop}

\begin{proof}
Fix $\alpha\in(0,\pi]$, $R\greq1$ and $\ve_1>0$. Let $\ve=\ve_1/K_{\alpha,R}$ where $K_{\alpha,R}$ is the constant  from Corollary \ref{cor:IntJR2}.

\emph{Step 1.}
There exists $0<r_\ve\leeq r_\textrm{z}$ such that
\begin{equation}\label{eq:frac}
   \bigg|\frac{(\Phi_\de^{-1})'(z)}{(\Phi_\de^{-1})'(w)}-1\bigg|<\ve,
\end{equation}
for $z,w\in B(0,r_\ve)$ and $\de\in B(0,r_\vartriangle)$.

If $z\in\mc N_{\de,R}$, then $-f_\de(z)$ is close to $p_\de$.
Let $n_\de$ be the largest number such that $\lambda_\de^{n_\de}\Phi_\de(-f_\de(z))\subset B(0,r_\ve)$ for all points $z\in\mc N_{\de,R}$, where $\alpha=\arg\de$.
Since $\diam(\mc N_{\de,R})\rightarrow0$ if $\de\rightarrow0$, we conclude that $n_\de\rightarrow\infty$.

Formula (\ref{eq:sumas}) for $m=n_\de+1$ leads to
\begin{equation}\label{eq:s}
   \dot\vp_\delta(s)=-\frac{1}{f_\delta'(\vp_\delta(s))}\bigg(1+\sum_{k=1}^{n_\de}
   \frac{1}{(f_\delta^{k})'(f_\de(\vp_\delta(s)))}\bigg)+\frac{\dot\vp_\delta(T^{n_\de+1}(s))}{(f_\delta^{n_\de+1})'(\vp_\delta(s))}.
\end{equation}

Now we estimate finite sum from the above formula.
If $\Phi(z)\in B(0,r_\ve/\lambda_\de^{n_\de})$ and $1\leeq k\leeq n_\de$, then we have (cf. (\ref{eq:pfp}))
$$f_\de^k(z)=\Phi_\de^{-1}(\lambda_\de^k\Phi_\de(z)),$$
and
$$(f_\de^k)'(z)=(\Phi_\de^{-1})'(\lambda_\de^k\Phi_\de(z))\cdot\lambda_\de^k\cdot\Phi_\de'(z).$$
Therefore, (\ref{eq:frac}) gives us
\begin{equation}\label{eq:la}
   \bigg|\frac{\lambda_\de^{k}}{(f_\de^{k})'(z)}-1\bigg|<\ve \quad\textrm{and}\quad\bigg|\frac{1}{(f_\de^{k})'(z)}-\frac{1}{\lambda_\de^{k}}\bigg|<\frac{\ve}{|\lambda_\de^{k}|}.
\end{equation}

Since $n_\de\rightarrow\infty$ and $\lambda_\de\rightarrow4$, we get
   $$\bigg|\sum_{k=1}^{n_\de}\frac{1}{\lambda_\de^{k}}-\frac13\bigg|<\ve.$$
Thus, (\ref{eq:la}) and the above estimate, lead to
   $$\bigg|\sum_{k=1}^{n_\de}\frac{1}{(f_\de^{k})'(z)}-\frac{1}{3}\bigg|<\ve+\sum_{k=1}^{n_\de}\frac{\ve}{|\lambda_\de^{k}|}<2\ve.$$

If $z\in\mc N_{\de,R}$, then $\Phi(-f_\de(z))\in B(0,r_\ve/\lambda_\de^{n_\de})$. Moreover $(f_\delta^{k})'(-f_\de(z))=-(f_\delta^{k})'(f_\de(z))$, so we have
\begin{equation}\label{eq:+1/3}
   \bigg|\sum_{k=1}^{n_\de}\frac{1}{(f_\de^{k})'(f_\de(z))}+\frac{1}{3}\bigg| <2\ve.
\end{equation}

\emph{Step 2.} Now we estimate integral of a "tail" from the formula (\ref{eq:s}).
Distortion of $f_\de^{n_\de+1}$ on $\mc N_{\de,R}$ is bounded by a constant depending on $\alpha$ and $R$. Since $\diam(\mc N_{\de,R})<K_1\sqrt{|\de|}$ and $\diam(f_\de^{n_\de+1}(\mc N_{\de,R}))>K_2\ve$, we get
   $$\big|(f_\de^{n_\de+1})'(z)\big|>K_3(\alpha,R,\ve)\frac{1}{\sqrt{|\de|}},$$
for $z\in \mc N_{\de,R}$. On the other hand $|z|>K_4\sqrt{|\de|}$ (cf. Proposition \ref{prop:Hmetric}), so if $\vp_\delta(s)=z\in \mc N_{\de,R}$, we obtain
   $$\bigg|\re\Big(\frac{v}{\varphi_\de(s)}\frac{\dot\vp_\delta(T^{n_\de+1}(s))}{(f_\delta^{n_\de+1})'(\vp_\delta(s))}\Big)\bigg| <K_5(\alpha,R,\ve)\big|\dot\vp_\delta(T^{n_\de+1}(s))\big|.$$
Therefore, Proposition \ref{prop:|phi|} leads to
\begin{equation}\label{eq:tail}
   |\de|^{1-\frac12d(\de)}\int_{N_{R}} \bigg|\re\Big(\frac{v}{\varphi_\de(s)}\frac{\dot\vp_\delta(T^{n_\de+1}(s))}{(f_\delta^{n_\de+1})'(\vp_\delta(s))}\Big) \bigg| d\tilde{\mu_{\de}}\rightarrow0.
\end{equation}

\emph{Step 3.}
We have $f'_\de(z)=2z$, so (\ref{eq:+1/3}) leads to
   $$\bigg|\re\bigg(\frac{v}{z}\cdot\frac{-1}{f_\delta'(z)}\bigg(1+\sum_{k=1}^{n_\de}
\frac{1}{(f_\delta^{k})'(f_\de(z))}\bigg)\bigg)-\re\bigg(-\frac{v}{2z^2}\frac23\bigg)\bigg|<\frac{\ve}{|z|^2}.$$
Next, we integrate both sides over $N_R$ (where $z=\vp_\de(s)$) with respect to $\tilde{\mu_\de}$, and multiply by $|\de|^{1-\frac12d(\de)}$.
So, formula (\ref{eq:s}) combined with (\ref{eq:tail}), Proposition \ref{prop:intJR} and Corollary \ref{cor:IntJR2}, gives us
   $$\bigg|\lim_{\de\rightarrow0}|\de|^{1-\frac12d(\de)}\int_{N_{R}}\re\Big(v\frac{\dot\vp_\de}{\varphi_\de}\Big)d\tilde{\mu_{\de}}- \frac13\frac2\pi I_{\alpha,R}\bigg|<\ve K_{\alpha,R}=\ve_1.$$
Because $\ve_1$ was arbitrary, the assertion follows.
\end{proof}

\begin{proof}[\textbf{Proof of the Theorem \ref{thm:minus2}}]
Fix $\alpha\in(0,\pi]$. We have (cf. (\ref{eq:I}), (\ref{eq:II}))
   $$\lim_{R\rightarrow\infty}I_{\alpha,R}=-{\sqrt6}\cos\alpha+ \frac{\sqrt6}{2}\sqrt{\sin\alpha}\int_{\alpha}^{\pi} \sqrt{\sin \beta}\,d\beta:= I_\alpha.$$
Let $\alpha=\arg\de$ and let $v=e^{i\alpha}$. Then, Propositions \ref{prop:intphiJR} and \ref{prop:int} lead to
\begin{equation}\label{eq:f2}
   \lim_{\de\rightarrow0}|\de|^{1-\frac{1}{2}d(\de)}\int_{\mc J_{\de_0}}\re\Big(v\frac{\dot\vp_{\de}}{\varphi_{\de}}\Big)d\tilde{{\mu}_{\de}}= \frac13\frac2\pi I_{\alpha}.
\end{equation}

Because $f_0$ is conjugated to $V$, the entropy $h_{\mu_0}(f_0)$ is equal to $\log2$. Next, we conclude from \cite[Theorem 11.4.1]{PU} that $1=\textrm{HD}(\mu_0)=h_{\mu_0}(f_0)/\chi_{\mu_0}(f_0)$ (where $\chi_{\mu_0}(f_0)$ is the Lyapunov exponent). Therefore $\chi_{\mu_0}(f_0)=\log2$, hence
   $$\int_{\mc J_0}\log|f_0'|d\mu_0=4\chi_{\mu_0}(f_0)=4\log2.$$
Since $f_\de'(\vp_\de)=2\vp_\de$, Propositions \ref{prop:uniform}, \ref{prop:mu}, give us
   $$\lim_{\de\rightarrow0}\int_{\mc J_{\de_0}}\log|f_\de'(\vp_\de)|d\tilde{\mu}_{\de}=4\log2.$$

Thus, formula (\ref{eq:d}) combined with (\ref{eq:f1}), (\ref{eq:f2}) and the above limit, leads to
\begin{equation}\label{eq:f3}
\lim_{\de\rightarrow0}|\de|^{1-\frac{1}{2}d(\de)}\cdot d_v'(\de)= -\frac{1}{6\pi\log 2} I_{\alpha},
\end{equation}
where $\alpha=\arg\de$ and let $v=e^{i\alpha}$.

In order to finish the proof, we have to replace $1-d(\de)/2$ by $1/2$ in the exponent. Because $I_\alpha$ is bounded, the above gives us
$$|d_v'(tv)|< K_1\cdot t^{\frac12d(tv)-1}.$$
Thus we have
\begin{multline*}
\big|d(tv)-1\big|\leeq\int_0^t\big|d_v'(sv)\big|\,ds< K_1\int_0^t s^{\frac12d(sv)-1}ds< K_1\int_0^t s^{-2/3}ds
\\= 3K_{1}\cdot t^{1/3}< t^{1/4},
\end{multline*}
and then
$$1\greq t^{|d(tv)-1|}\greq t^{t^{1/4}}=e^{t^{1/4}\log t}\rightarrow e^0=1.$$
So, we obtain
$$\lim_{t\rightarrow0^+}\frac{t^{1-\frac12d(tv)}}{t^{\frac12}}=\lim_{t\rightarrow0^+}t^{\frac{1}{2}(1-d(tv))}=1,$$
and finally (\ref{eq:f3}) and (\ref{eq:omega}) lead to
\begin{equation*}
\lim_{\de\rightarrow0}|\de|^{\frac12}\, d_v'(\de)= -\frac{1}{6\pi\log 2} I_{\alpha}=\Omega_{-2}(\alpha),
\end{equation*}
where $\alpha=\arg\de$ and let $v=e^{i\alpha}$.
\end{proof}

%$$|(f^{-n-1}_\de)'|^{d(\de)}=|(f^{-n-1}_\de)'|^{d(\de)-1}\cdot|(f^{-n-1}_\de)'|,\quad |(f^{-n-1}_\de)'|\asymp\sqrt{|\de|}$$
%$$\Big(\frac{|(f^{-n-1}_\de)'|}{\sqrt{|\de|}}\Big)^{d(\de)-1}\rightarrow1, \quad |(f^{-n-1}_\de)'|^{d(\de)-1}\approx|\de|^{\frac12(d(\de)-1)}$$

%$$\frac{\tilde{\mu}^\updownarrow_{\delta,R}(A)}{|\de|^{\frac12(d(\de)-1)}}\approx \frac{1}{\sqrt{|\de|}}\int_{f_\de^{n+1}(|\de|A)}|(f^{-n-1}_\de)'|d\mu_\de\rightarrow\tilde{\mu}^\updownarrow_{0,R}(A)$$


\begin{thebibliography}{WW}

%\bibitem{ADU} J. Aaronson, M. Denker, M. Urba\'{n}ski, Ergodic theory for Markov fibered systems and parabolic rational maps. Trans. Amer. Math. Soc. \textbf{337}, 495--548 (1993)

%\bibitem{BZ} O. Bodart, M. Zinsmeister, Quelques r\'{e}sultats sur la dimension de Hausdorff des ensembles de Julia des polyn\^{o}mes quadratiques. Fund. Math. \textbf{151}(2), 121--137 (1996)

%\bibitem{BT} X. Buff, Tan Lei, Dynamical convergence and polynomial vector fields. J. Differential Geom. \textbf{77}, 1--41 (2007)

\bibitem{CG} L. Carleson, T.W. Gamelin, Complex Dynamics, Springer-Verlag, New York, 1993.

\bibitem{CK} Y.-C. Chen, T. Kawahira, From Cantor to semi-hyperbolic parameter along external rays, Trans. Amer. Math. Soc. \textbf{372} (2019) 7959--7992. https://doi.org/10.1090/tran/7839.

%\bibitem{DU} M. Denker, M. Urbański, Hausdorff measures on Julia sets of subexpanding rational maps. Israel J. Math. \textbf{76}, 193--214 (1991)

\bibitem{DGM} N. Dobbs, J. Graczyk, N. Mihalache, Hausdorff Dimension of Julia sets in the logistic family, arXiv:2007.07661.


\bibitem{D} A. Douady, Does a Julia set depend continuously on the polynomial?, in: R. Devaney (ed.), Complex dynamical systems, Proc. Sympos. Appl. Math. \textbf{49}, Amer. Math. Soc., 1994, pp. 91--135. https://doi.org/10.1090/psapm/049/1315535.

%\bibitem{DSZ} A. Douady, P. Sentenac, M. Zinsmeister, Implosion parabolique et dimension de Hausdorff. C. R. Acad. Sci. Paris, Ser. I, \textbf{325}, 765--772 (1997)

\bibitem{DGT} A. Dudko, I. Gorbovickis, W. Tucker, Lower bounds on the Hausdorff dimension of some Julia sets, arXiv:2204.07880.

\bibitem{FJW} A. Fan, Y. Jiang, J. Wu, Asymptotic Hausdorff dimensions of Cantor sets associated with an asymptotically non-hyperbolic family, Ergod. Th. \& Dynam. Sys. \textbf{25} (2005) 1799-1808.  https://doi.org/10.1017/S014338570500009X.

%\bibitem{HZ} G. Havard, M. Zinsmeister, Le chou-fleur a une dimension de Hausdorff inf\'{e}rieure \`{a} 1,295. Preprint (2000)
    %(available at http://www.univ-orleans.fr/mapmo/publications/mapmo2000/mapmo0001.pdf; 2000)

\bibitem{HZi} G. Havard, M. Zinsmeister, Thermodynamic formalism and variations of the Hausdorff dimension of quadratic Julia sets, Commun. Math. Phys. \textbf{210} (2000) 225--247. https://doi.org/10.1007/s002200050778.

\bibitem{Hub} J. H. Hubbard, Teichm\"uller Theory and Applications to Geometry, Topology and Dynamics. vol. 1. Teichm\"uller Theory, Matrix Editions, Ithaca, NY. 2006.

\bibitem{J} L. Jaksztas, On the derivative of the Hausdorff dimension of the quadratic Julia sets, Trans. Amer. Math. Soc. \textbf{363} (2011) 5251--5291. https://doi.org/10.1090/S0002-9947-2011-05208-6.

\bibitem{Jiii} L. Jaksztas, On the directonal derivative of the Hausdorff dimension of quadratic polynomial Julia sets at $1/4$, Nonlinearity \textbf{33} (2020) 5919--5960. https://doi.org/10.1088/1361-6544/ab9a1a.

%\bibitem{Ji} L. Jaksztas, On the left side derivative of the Hausdorff dimension of the Julia sets for $z^2+c$ at $c=-3/4$. Israel J.
%Math. \textbf{199}(1), 1--43 (2014)

%\bibitem{Jii} L. Jaksztas, The Hausdorff dimension is convex on the left side of $1/4$. Proc. Edinb. Math. Soc. \textbf{60}(4), 911--936 (2017)


%\bibitem{JZ} L. Jaksztas, M. Zinsmeister, On the derivative of the Hausdorff dimension of the Julia sets for $z^2+c$, $c\in\R$ at satellite parabolic parameters. Preprint. arXiv:1712.03102

%\bibitem{Mii} C. McMullen, Hausdorff dimension and conformal dynamics II: Geometrically finite rational maps. Comm. Math. Helv. \textbf{75}, 535--593 (2000)
%\bibitem{Miii} C. McMullen, Hausdorff dimension and conformal dynamics III: Computation of dimension. Amer. J. Math. \textbf{120}, 471--515 (1998)
\bibitem{MS} W. de Melo, S. van Strien, One-Dimensional Dynamics, Springer-Verlag, Berlin, 1993.

\bibitem{PU} F. Przytycki, M. Urbanski, Fractals in the Plane, Ergodic Theory Methods, Cambridge University press, vol. 371, 2010.

%\bibitem{Ra} T. Ransford, Potential Theory in the Complex Plane. London Mathematical Society Student Texts 28, (1995)

\bibitem{RL} J. Rivera-Letelier, On the continuity of Hausdorff dimension of Julia sets and similarity between the Mandelbrot set and Julia sets, Fund. Math. \textbf{170} (2001) 287--317. https://doi.org/10.4064/fm170-3-6.
\bibitem{R} D. Ruelle, Repellers for real analytic maps, Ergod. Th. \& Dynam. Sys. \textbf{2} (1982) 99--107. https://doi.org/10.1017/S0143385700009603.
%\bibitem{S} F. Schweiger, Numbertheoretical endomorphisms with $\sigma$-finite invariant measure. Israel J. Math. \textbf{21}, 308--318 (1975)
%\bibitem{Zd} A. Zdunik, Parabolic orbifolds and the dimension of the maximal measure for rational maps. Invent. Math. \textbf{99}, 627--649 (1990)

\bibitem{Z} M. Zinsmeister, Thermodynamic Formalism and Holomorphic Dynamical Systems, SMF/AMS Texts and Monographs, vol. 2, 1996.

\end{thebibliography}
\end{document}